\documentclass[12pt,twoside]{amsart}
\usepackage{amssymb,amsmath,amsthm, amscd, enumerate, mathrsfs}
\usepackage{graphicx, hhline}
\usepackage[all]{xy}
\usepackage[usenames]{color}
\usepackage{hyperref}
\usepackage{fancyhdr}
\hypersetup{colorlinks=true}

\title{Remarks on special kinds of the relative log minimal model program}
\author{Kenta Hashizume}
\date{2017/4/5, version 0.04}
\address{Department of Mathematics, Graduate School of Science, 
Kyoto University, Kyoto 606-8502, Japan}
\email{hkenta@math.kyoto-u.ac.jp}

\newtheorem{thm}{Theorem}[section]
\newtheorem{lem}[thm]{Lemma}
\newtheorem{cor}[thm]{Corollary}
\newtheorem{prop}[thm]{Proposition}

\theoremstyle{definition}
\newtheorem{defn}[thm]{Definition}
\newtheorem{rem}[thm]{Remark}
\newtheorem*{ack}{Acknowledgments} 
\newtheorem{say}[thm]{}
\newtheorem{step}{Step}
\newtheorem{step2}{Step}
\newtheorem{step3}{Step}
\newtheorem{step4}{Step}
\newtheorem{step5}{Step}
\newtheorem{step6}{Step}

\begin{document}

\maketitle

\begin{abstract}
We prove $\mathbb{R}$-boundary divisor versions of results proved by Birkar \cite{birkar-flip} or Hacon--Xu \cite{haconxu-lcc} on special kinds of the relative log minimal model program.
\end{abstract}
\tableofcontents

\section{Introduction}   
We will work over $\mathbb{C}$, the complex number field.% and ``log canonical pair'' means a log canonical pair with a boundary $\mathbb{R}$-divisor.   

The main results of this article are the following theorems, which are,
%Theorem \ref{thm1.3} and Theorem \ref{thmhaconxu} below. 
roughly speaking, $\mathbb{R}$-divisor versions of \cite[Theorem 1.1]{birkar-flip} and \cite[Theorem 1.1]{haconxu-lcc} respectively:

\begin{thm}[cf.~{\cite[Theorem 1.1]{birkar-flip}}]\label{thm1.3}
Let $\pi:X \to Z$ be a projective morphism of normal quasi-projective varieties and let $(X,B+A)$ be a log canonical pair, where $B \geq 0$ is an $\mathbb{R}$-divisor and $A \geq 0$ is $\mathbb{R}$-Cartier, such that $K_{X}+B+A\sim_{\mathbb{R},\,Z}0$. 
Suppose that $K_{X}+B$ is pseudo-effective over $Z$. 

Then $(X,B)$ has a good minimal model over $Z$. 
\end{thm}

%Theorem \ref{thm1.3} is the $\mathbb{R}$-divisor version of  \cite[Theorem 1.1]{birkar-flip}.
%By using Theorem \ref{thm1.3} we can prove the following result: 

\begin{thm}[cf.~{\cite[Theorem 1.1]{haconxu-lcc}}]\label{thmhaconxu}
Let $\pi:X\to Z$ be a projective morphism of normal quasi-projective varieties 
and let $(X,\Delta)$ be a log canonical pair with a boundary $\mathbb{R}$-divisor $\Delta$. 
Assume that there exists an open subset $U\subset Z$ such that the pair $(\pi^{-1}(U),\Delta\!\!\mid_{\pi^{-1}(U)})$ has a good minimal model over $U$ and any lc center of $(X,\Delta)$ intersects $\pi^{-1}(U)$. 

 Then $(X,\Delta)$ has a good minimal model over $Z$. 
\end{thm}

% Proof of the theorem is described in Section \ref{sec3}. 
%Roughly speaking, Theorem \ref{thm1.3} and Theorem \ref{thmhaconxu} is the $\mathbb{R}$-divisor version of \cite[Theorem 1.1]{birkar-flip} and \cite[Theorem 1.1]{haconxu-lcc} respectively. 

%The difficulty of the proof of main results is that we can not directly apply the finite generation of log canonical rings for klt pairs and its applications (see \cite[Section 5]{birkar-flip}, \cite{cortilazic} and \cite[Section 2]{haconxu-lcc}). 
If $B,\,A$ and $\Delta$ are $\mathbb{Q}$-divisors in Theorem \ref{thm1.3} and Theorem \ref{thmhaconxu}, those theorems are nothing but \cite[Theorem 1.1]{birkar-flip} and \cite[Theorem 1.1]{haconxu-lcc} respectively. 
The proof of those theorems are described in Section \ref{sec3}. 

In general, theorems for lc pairs with $\mathbb{Q}$-boundary divisors can not 
%always 
be generalized directly to $\mathbb{R}$-boundary divisor case.  
One of the most important differences between those two cases is that in  $\mathbb{Q}$-boundary divisor case we can consider the finite generation of log canonical rings of the lc pairs. 
This difference is especially important when we deal with klt pairs. 
In fact, today the finite generation of log canonical rings for klt pairs yields various applications (see \cite[Section 5]{birkar-flip}, \cite{cortilazic} 
and \cite[Section 2]{haconxu-lcc}). 
In \cite{birkar-flip} and \cite{haconxu-lcc}, 
some 
applications of the finite generation of log canonical rings for klt pairs play crucial roles in the proof of \cite[Theorem 1.1]{birkar-flip} and \cite[Theorem 1.1]{haconxu-lcc}. 
On the other hand, when we deal with lc pairs with $\mathbb{R}$-boundary divisors, we cannot apply the argument as in \cite[Section 5]{birkar-flip} or \cite[Section 2]{haconxu-lcc}. 
So we need to seek other approaches to prove the main results. 
In this article we apply the argument of weak semi-stable reduction 
%which was 
developed by Abramovich and Karu \cite{ak}. 
By this argument %of weak semi-stable reduction
 (cf.~\cite[Theorem 2.1]{ak}) and Ambro's canonical bundle formula 
 for certain lc pairs
  (cf.~ \cite[Corollary 3.2]{fg-bundle}), an inductive argument works on fiber spaces $(X,\Delta)\to V$
  on which the relative numerical dimension of $K_{X}+\Delta$
  %the log canonical divisor 
  is zero and all lc centers 
  of $(X,\Delta)$
   dominate $V$
   %the base variety 
   (see Proposition \ref{prop2.1}). 
This result plays a crucial role in the proof of the main results. 
For details, see Section \ref{sec2}.

We also would like to remark that in our proof of the main results we do not apply \cite[Theorem 1.1]{birkar-flip} or \cite[Theorem 1.1]{haconxu-lcc} directly. 
The matter is the abundance theorem in those situations. 
In \cite{birkar-flip} and \cite{haconxu-lcc}, they apply \cite{haconxu}, which is the relative version of  \cite{fujino-gongyo}. 
In particular, in the situation of \cite[Theorem 1.1]{birkar-flip}%(or Theorem \ref{thm1.3})
, we can not take compactifications of given varieties to use \cite{fujino-gongyo} keeping the hypothesis of the theorem. 
So we cannot avoid applying \cite{haconxu} instead of \cite{fujino-gongyo}. 
In this article we give proof of the main results using compactifications and \cite{fujino-gongyo}. 
First we prove Theorem \ref{thmhaconxu}. 
In Theorem \ref{thmhaconxu} we can take compactifications of $X$ and $Z$ keeping the hypothesis of the theorem. 
For details, see Subsection \ref{subsec4.2}. 
After that, we prove a generalization of Theorem \ref{thm1.3} using  Theorem \ref{thmhaconxu} and \cite{fujino-gongyo}. 
For details, see Subsection \ref{subsec4.25}.

With Theorem \ref{thm1.3} and Theorem \ref{thmhaconxu}, we can prove $\mathbb{R}$-divisor versions of some results which are known in $\mathbb{Q}$-boundary divisor case: 

\begin{cor}[existence of lc closures, cf.~{\cite[Corollary 1.2]{haconxu-lcc}}]\label{corlcc}
Let $U^{0}$ be an open subset of a normal quasi-projective variety $U$, $f^{0}:X^{0}\to U^{0}$ be a projective morphism, and $(X^{0},\Delta^{0})$ be a log canonical pair. 
Then there exists a projective morphism $f:X\to U$ and a log canonical pair $(X,\Delta)$ such that $X^{0}=f^{-1}(U^{0})$ is an open subset and $\Delta^{0}=\Delta\!\!\mid_{X^{0}}$. 
Moreover, any lc center of $(X,\Delta)$ intersects $X^{0}$. 
\end{cor}

\begin{thm}[existence of lc flips, cf.~{\cite[Corollary 4.8.12]{fujino-book}}]\label{corflip}
Let $(X,\Delta)$ be a log canonical pair and $\varphi:X \to W$ be a flipping contraction, that is, a projective birational morphism  
of normal varieties such that 
\begin{enumerate}
\item[$\bullet$]
$X$ and $W$ are isomorphic in codimension one, and 
\item[$\bullet$]
$-(K_{X}+\Delta)$ is $\varphi$-ample. 
\end{enumerate}
Then $(K_{X}+\Delta)$-flip of $\varphi$ exists. 
\end{thm}

\begin{thm}[cf.~{\cite[Theorem 1.1]{has-trivial}}]\label{corhas}
Let $\pi:X \to Z$ be a projective surjective morphism of normal projective varieties and let $(X,\Delta)$ be a log canonical pair.  
Suppose that $K_{X}+\Delta \sim_{\mathbb{R}}\pi^{*}D$ for an $\mathbb{R}$-divisor $D$ on $Z$.  
Fix $d\in \mathbb{Z}_{>0}$ and assume that all $d$-dimensional Kawamata log terminal pairs have a good minimal model or a Mori fiber space. 
If 
\begin{itemize}
\item
${\rm dim}\,Z\leq d$, or
\item
${\rm dim}\,Z=d+1$ and $D$ is big, 
\end{itemize}
then $(X,\Delta)$ has a good minimal model or a Mori fiber space.
\end{thm}

\begin{cor}\label{cor1.4}
Let $(X,\Delta)$ be a projective log canonical pair such that $K_{X}+\Delta$ is abundant and $0\leq\nu(X,K_{X}+\Delta)\leq 4$. 

Then $(X,\Delta)$ has a good minimal model. 
\end{cor}
For the definition of (relatively) abundant $\mathbb{R}$-Cartier $\mathbb{R}$-divisors, see Section \ref{sec1}. 
We note that in Theorem \ref{corflip} we do not assume the relative picard number is one. 
We hope that the main results and those other results have important applications in the minimal model theory.

The contents of this article are as follows: 
In Section \ref{sec1} we collect some notations and definitions, and some lemmas on existence of log minimal models. 
In Section \ref{sec2} we discuss the log MMP on fiber spaces on which the relative numerical dimension of the log canoincal divisor is zero. 
In Section \ref{sec3} we prove Theorem \ref{thm1.3}, Theorem \ref{thmhaconxu} and other results. 
%In the proof of Theorem \ref{thm1.3} in Subsection \ref{subsec4.1} we  directly apply \cite[Theorem 1.1]{birkar-flip}. In Section \ref{secapp}, which is the appendix, we prove a generalization of Theorem \ref{thm1.3} without  \cite[Theorem 1.1]{birkar-flip} or \cite[Theorem 1.1]{haconxu-lcc}. As a direct consequence, we see that we can in fact prove Theorem \ref{thm1.3} without  \cite[Theorem 1.1]{birkar-flip}. 

\begin{ack}
The author was partially supported by JSPS KAKENHI Grant Number JP16J05875 from JSPS. 
The author would like to express his gratitude to Professor Osamu Fujino for much useful advice, discussions, and answering his questions.
\end{ack}

\section{Preliminaries}\label{sec1}

\subsection{Notations and definitions}
%In this subsection we collect some notations and definitions. 
We will freely use the notations and definitions in \cite{bchm}. 

\begin{say}[Maps]
Let $f:X\to Y$ be a projective morphism from a normal variety to a variety.  
Then $f$ is a {\em contraction} if $f$ is surjective and it has connected fibers. 

Let $f:X\dashrightarrow Y$ be a birational map of normal varieties. 
Then $f$ is a {\em birational contraction} if $f^{-1}$ does not contract any divisors. 
Let $D$ be an $\mathbb{R}$-divisor on $X$. 
Unless otherwise stated, we mean $f_{*}D$ by denoting $D_{Y}$ or $(D)_{Y}$. 
\end{say}

\begin{say}[Relative numerical dimension and relatively abundant $\mathbb{R}$-divisors]\label{numeri}
Let $\pi:X\to Z$ be a projective surjective morphism from a normal variety to a variety with connected fibers, and let $F$ be a very general fiber of $\pi$. 
For any $\mathbb{R}$-Cartier $\mathbb{R}$-divisor $D$ on $X$, we denote $\nu(F,D\!\!\mid_{F})$ by $\nu(X/Z, D)$. 

Notations as above, we define the {\em invariant Iitaka dimension} of $D\!\!\mid_{F}$, which we denote $\kappa_{\iota}(F,D\!\!\mid_{F})$, as follows (see \cite[Definition 2.5.5]{fujino-book} and \cite[Definition 2.2]{gongyo1}): 
If there exists $E\geq 0$ such that $D\sim_{\mathbb{R},\,Z}E$, then we set $\kappa_{\iota}(F,D\!\!\mid_{F})=\kappa(F,E\!\!\mid_{F})$. 
Here right hand side is the usual Kodaira dimension of $E\!\!\mid_{F}$. 
Otherwise we set $\kappa_{\iota}(F,D\!\!\mid_{F})=-\infty$. 
%We can check that $\kappa_{\iota}(F,D\!\!\mid_{F})$ is well-defined. 

With the definition, $D$ is {\em relatively abundant} (or {\em abundant over} $Z$) if the equality $\nu(X/Z, D)=\kappa_{\iota}(F,D\!\!\mid_{F})$ holds. 
% where $\kappa_{\iota}(F,D\!\!\mid_{F})$ is the invariant Iitaka dimension of $D\!\!\mid_{F}$. 
%For the definition of invariant Iitaka dimension, see \cite[Definition 2.5.5]{fujino-book} and \cite[Definition 2.2]{gongyo1}. 
\end{say}

We introduce basic properties of relative numerical dimension. 
These follow from properties of usual numerical dimension.

\begin{lem}[cf.~{\cite[Lemma 2.2]{gongyo1}} and {\cite[V, 2.7 Proposition]{nakayama-zariski-decom}}]\label{lemnum}
Let $\pi:X\to Z$ and $D$ be as above. 
\begin{enumerate}
\item
$D$ is big over $Z$ if and only if $\nu(X/Z,D)={\rm dim}\,X-{\rm dim}\,Z$. 
\item
Let $D_{1}$ and $D_{2}$ be two $\mathbb{R}$-Cartier $\mathbb{R}$-divisors on $X$. 
Suppose that $D_{1}\sim_{\mathbb{R},\,Z}E_{1}$ and $D_{2}\sim_{\mathbb{R},\,Z}E_{2}$ for effective divisors $E_{1}$ and $E_{2}$ such that ${\rm Supp}\,E_{1}={\rm Supp}\,E_{2}$. 
Then $\nu(X/Z,D_{1})=\nu(X/Z, D_{2})$. 
\item
Let $f:Y \to X$ be a projective birational morphism from a normal variety and $D'$ be an $\mathbb{R}$-Cartier $\mathbb{R}$-divisor on $Y$ such that $D'$ is the sum of $f^{*}D$ and an effective $f$-exceptional divisor. 
Then $\nu(Y/Z,D')=\nu(X/Z, D)$.
\end{enumerate}
\end{lem}

\begin{say}[Singularities of pairs]
A {\em pair} $(X,\Delta)$ consists of a normal variety $X$ and a boundary $\mathbb{R}$-divisor, that is, an $\mathbb{R}$-divisor whose coefficients belong to $[0,1]$, on $X$ such that $K_{X}+\Delta$ is $\mathbb{R}$-Cartier. 

Let $(X, \Delta)$ be a pair and $f:Y\to X$ be a log resolution of $(X, \Delta)$. 
Then we can write 
$$K_Y=f^*(K_X+\Delta)+\sum_{i} a(E_{i}, X,\Delta) E_i$$ 
where $E_{i}$ are prime divisors on $Y$ and $a(E_{i}, X,\Delta)$ is a real number for any $i$. 
Then we call $a(E_{i}, X,\Delta)$ the {\em discrepancy} of $E_{i}$ with respect to $(X,\Delta)$. 
The pair $(X, \Delta)$ is called {\it Kawamata log terminal} ({\it klt}, for short) if $a(E_{i}, X, \Delta) > -1$ for any log resolution $f$ of $(X, \Delta)$ and any $E_{i}$ on $Y$. 
$(X, \Delta)$ is called {\it log canonical} ({\it lc}, for short) if $a(E_{i}, X, \Delta) \geq -1$ for any log resolution $f$ of $(X, \Delta)$ and any $E_{i}$ on $Y$. 
$(X, \Delta)$ is called {\it divisorial log terminal} ({\it dlt}, for short) if  there exists a log resolution $f:Y \to X$ of $(X, \Delta)$ such that $a(E, X, \Delta) > -1$ for any $f$-exceptional prime divisor $E$ on $Y$. 
If $(X,\Delta)$ is a log canonical pair and $P$ is a prime divisor over $X$ such that $a(P,X,\Delta)=-1$, then the image of $P$ on $X$ is called an {\em lc center} of $(X,\Delta)$. 
\end{say}

Next we define some models. 
Note that our definition of log minimal models is slightly different from that of \cite%[Definition 2.1]
{birkar-flip}. 
%In \cite{birkar-flip}, log minimal models are supposed to be dlt. 
%i.e.,
The difference is we do not assume that log minimal models are dlt. 
But this difference is intrinsically not important (cf.~\cite[Remark 2.7]{has-trivial}). 
In our definition, any weak lc model $(X',\Delta')$ (see Definition \ref{deflogmin} below) of a $\mathbb{Q}$-factorial lc pair $(X,\Delta)$ constructed with the $(K_{X}+\Delta)$-MMP
%obtained by running the $(K_{X}+\Delta)$-MMP on $\mathbb{Q}$-factorial lc pairs are
is a log minimal model of $(X,\Delta)$ even though $(X',\Delta')$ may not be dlt.

\begin{defn}[cf.~{\cite[Definition 2.5]{has-trivial}}, see also {\cite[Definition 2.1]{birkar-flip}}]\label{defmin}\label{deflogmin}
Let $\pi:X \to Z$ be a projective morphism from a normal variety to a variety and let $(X,\Delta)$ be a log canonical pair.  
Let $\pi ':X' \to Z$ be a projective morphism from a normal variety to $Z$ and $\phi:X \dashrightarrow X'$ be a birational map over $Z$. 
Let $E$ be the reduced $\phi^{-1}$-exceptional divisor on $X'$, that is, $E=\sum E_{j}$ where $E_{j}$ are $\phi^{-1}$-exceptional prime divisors on $X'$. 
Then the pair $(X', \Delta'=\phi_{*}\Delta+E)$ is called a {\em log birational model} of $(X,\Delta)$ over $Z$. 
A log birational model $(X', \Delta')$ of $(X,\Delta)$ over $Z$ is a {\em weak log canonical model} ({\em weak lc model}, for short) if 
\begin{itemize}
\item
$K_{X'}+\Delta'$ is nef over $Z$, and 
\item
for any prime divisor $D$ on $X$ which is exceptional over $X'$, we have
$$a(D, X, \Delta) \leq a(D, X', \Delta').$$ 
\end{itemize}
A weak lc model $(X',\Delta')$ of $(X,\Delta)$ over $Z$ is a {\em log minimal model} if 
\begin{itemize}
\item
$(X',\Delta')$ is $\mathbb{Q}$-factorial, and 
\item
the above inequality on discrepancies is strict. 
\end{itemize}
A log minimal model $(X',\Delta')$ of $(X, \Delta)$ over $Z$ is called a {\em good minimal model} if $K_{X'}+\Delta'$ is semi-ample over $Z$.    
\end{defn}

\begin{defn}[Log smooth models, cf.~{\cite[Definition 2.3]{birkar-flip}}]\label{deflogsm}
Let $(X,\Delta)$ be a log canonical pair and $f:Y \to X$ be a log resolution of $(X,\Delta)$. 
Let $\Gamma$ be a boundary $\mathbb{R}$-divisor on $Y$ such that $(Y,\Gamma)$ is log smooth. 
Then $(Y,\Gamma)$ is a {\em log smooth model} of $(X,\Delta)$ if we can write 
$$K_{Y}+\Gamma=f^{*}(K_{X}+\Delta)+F$$
with an effective $f$-exceptional divisor $F$ such that
\begin{enumerate}
\item[$\bullet$] 
every $f$-exceptional prime divisor $E$ satisfying $a(E,X,\Delta)>-1$ is a component of $F$ and $\Gamma-\llcorner \Gamma \lrcorner$.  
\end{enumerate}
%When $\Delta$ is a $\mathbb{Q}$-divisor and $f:Y \to X$ is a log resolution of $(X,\Delta)$, we can find a $\mathbb{Q}$-divisor $\Gamma$ on $Y$ such that $(Y,\Gamma)$ is a log smooth model of $(X,\Delta)$. 
\end{defn}

\subsection{Properties and remarks on models}
Here we mainly introduce sufficient or equivalent conditions on existence of log minimal models. 

Let $\pi:X \to Z$ be a projective morphism of normal quasi-projective varieties and let $(X,\Delta)$ be a log canonical pair.  
By Definition \ref{defmin} all log minimal models are weak lc models. 
Furthermore we see that:

\begin{rem}[cf.~{\cite[Corollary 3.7]{birkar-flip}}]\label{thmlogmin}
If $(X,\Delta)$ has a weak lc model over $Z$, then $(X,\Delta)$ has a log minimal model over $Z$.
\end{rem}

%We have some remarks of those models (cf.~\cite{birkar-flip}).

\begin{rem}[cf.~{\cite[Remark 2.8]{birkar-flip}}]\label{remlogsm}
To find a log minimal model of $(X,\Delta)$ over $Z$, we can replace $(X,\Delta)$ with its log smooth model. 
\end{rem}

\begin{rem}[cf.~{\cite[Remark 2.7]{birkar-flip}}]\label{remtwoweak}
%Let $\pi:X \to Z$ be a projective morphism from a normal variety to a variety and let $(X,\Delta)$ be a log canonical pair.  
Let $(X',\Delta')$ and $(X'',\Delta'')$ be two weak lc models of $(X,\Delta)$ over $Z$ and let $g':W\to X'$ and $g'':W\to X''$ be a common resolution of the induced birational map $X'\dashrightarrow X''$. 
Then we have 
$$g'^{*}(K_{X'}+\Delta')=g''^{*}(K_{X''}+\Delta'').$$
In particular, if there is a weak lc model of $(X,\Delta)$ over $Z$ with relatively semi-ample log canonical divisor, then any log canonical divisor of a weak lc model of $(X,\Delta)$ is relatively semi-ample. 
\end{rem}

By combining Remark \ref{thmlogmin}, Remark \ref{remlogsm} and Remark \ref{remtwoweak}, we have:
 
\begin{lem}\label{lemweakmin}
To show the existence of  good minimal models of $(X,\Delta)$ over $Z$, it is sufficient to find a weak lc model of $(X,\Delta)$ over $Z$ with relatively semi-ample log canonical divisor.
Moreover we can freely take log smooth models or dlt blow-ups and replace $(X,\Delta)$ by those models.  
\end{lem}

Next we state properties of log smooth models. 
%Let $\pi:X\to Z$ and $(X,\Delta)$ be as above.

\begin{rem}%[Properties of log smoothmodels]
$f:(Y,\Gamma)\to (X,\Delta)$ be a log smooth model of a log canonical pair $(X,\Delta)$, and let $F$ be an effective $f$-exceptional divisor as in Definition \ref{deflogsm}. 
By Definition \ref{deflogsm} we see that 
\begin{itemize}
\item
${\rm Supp}\,\Gamma ={\rm Supp}\,f_{*}^{-1}\Delta \cup {\rm Ex}\,(f)$, 
\item
the image of any lc center of $(Y,\Gamma)$ on $X$ is an lc center of $(X,\Delta)$, and  
\item
for any $f$-exceptional prime divisor $E$, $E$ is a component of $F$ if and only if $a(E,X,\Delta)>-1$. 
\end{itemize}
\end{rem}

Here we introduce an important theorem by Birkar \cite{birkar-flip} and a lemma about the log MMP with scaling.

\begin{thm}[cf.~{\cite[Theorem 4.1]{birkar-flip}}]\label{thmtermi}
Let $(X,\Delta)$ be a $\mathbb{Q}$-factorial log canonical pair such that $(X,0)$ is klt, and let $\pi:X \to Z$ be a projective morphism of normal quasi-projective varieties. 
If there exists a log minimal model of $(X,\Delta)$ over $Z$, then any $(K_{X}+\Delta)$-MMP over $Z$ with scaling of an ample divisor terminates. 
\end{thm}

%By Remark \ref{thmlogmin} and Theorem \ref{thmtermi}, existence of a weak lc model of $(X,\Delta)$ over $Z$ is equivalent to termination of any $(K_{X}+\Delta)$-MMP over $Z$ with scaling of an ample divisor.

\begin{lem}\label{lemmmp}
Let $(X,\Delta)$ be a $\mathbb{Q}$-factorial log canonical pair such that $(X,0)$ is klt, and let $\pi:X \to Z$ be a projective morphism of normal quasi-projective varieties.  
Let $H\geq 0$ be an $\mathbb{R}$-divisor on $X$ such that $(X,\Delta+H)$ is log canonical and $K_{X}+\Delta+H$ is nef over $Z$. 
Assume that for any $0<\mu\leq1$ the pair $(X,\Delta+\mu H)$ has a log minimal model over $Z$.

Then we can construct a sequence of steps of the $(K_{X}+\Delta)$-MMP over $Z$ with scaling of $H$ 
$$(X=X^{0},\Delta=\Delta_{X^{0}})\dashrightarrow\cdots \dashrightarrow (X^{i},\Delta_{X^{i}})\dashrightarrow \cdots$$
such that if we set
$\lambda_{i}={\rm inf}\{\nu\geq0 \mid K_{X^{i}}+\Delta_{X^{i}}+\nu H_{{X^{i}}}{\rm \;is\;nef\;over\;}Z\},$ 
then the $(K_{X}+\Delta)$-MMP terminates after finitely many steps or we have ${\rm lim}_{i \to \infty}\lambda_{i}=0$ even if the $(K_{X}+\Delta)$-MMP does not terminates. 
\end{lem}

\begin{proof}
%Set $\lambda_{0}=\{\nu\geq0 \mid K_{X}+\Delta+\nu H{\rm \;is\;nef\;over\;}Z\}$. 
Throughout the proof let $\lambda_{i}$ be as in the lemma. 
If $\lambda_{0}=0$, there is nothing to prove. 
So assume $\lambda_{0}>0$. 
Pick a positive real number $\lambda'_{0}<\lambda_{0}$ sufficiently close to $\lambda_{0}$,  and run the $(K_{X}+\Delta+\lambda'_{0}H)$-MMP over $Z$ with scaling of an ample divisor. 
By Theorem \ref{thmtermi} and hypothesis, this MMP terminates with a log minimal model 
$$(X,\Delta+\lambda'_{0}H)\dashrightarrow (X^{k_{1}},\Delta_{X^{k_{1}}}+\lambda'_{0}H_{X^{k_{1}}})$$
over $Z$. 
Since we chose $\lambda'_{0}$ sufficiently close to $\lambda_{0}$, for any $0\leq i<k_{1}$, $K_{X^{i}}+\Delta_{X^{i}}+\lambda_{0}H_{X^{i}}$ is trivial over the extremal contraction. 
Then we see that the birational map $(X,\Delta)\dashrightarrow (X^{k_{1}},\Delta_{X^{k_{1}}})$ is a finitely many steps of the $(K_{X}+\Delta)$-MMP over $Z$ with scaling of $H$, and we have $\lambda_{0}= \lambda_{1}= \cdots =\lambda_{k_{1}-1}$. 
Moreover $\lambda_{0}> \lambda'_{0}\geq\lambda_{k_{1}}$ 
by construction, 
 and hence $\lambda_{k_{1}-1}>\lambda_{k_{1}}$. 
If $\lambda_{k_{1}}=0$ then we stop constructing the sequence of steps of the $(K_{X}+\Delta)$-MMP over $Z$. 
If $\lambda_{k_{1}}>0$ then pick a positive real number $\lambda'_{1}<\lambda_{k_{1}}$ sufficiently close to $\lambda_{k_{1}}$, and run the $(K_{X^{k_{1}}}+\Delta_{X^{k_{1}}}+\lambda'_{1}H_{X^{k_{1}}})$-MMP over $Z$ with scaling of an ample divisor. 

By repeating the above discussion, we can construct a sequence of steps of the $(K_{X}+\Delta)$-MMP over $Z$ with scaling of $H$ 
$$(X=X^{0},\Delta=\Delta_{X^{0}})\dashrightarrow\cdots \dashrightarrow (X^{k_{i}},\Delta_{X^{k_{i}}})\dashrightarrow \cdots$$
such that $\lambda_{0}\geq\cdots \geq\lambda_{k_{1}-1}>\lambda_{k_{1}}\geq\cdots \geq\lambda_{k_{i}-1}>\lambda_{k_{i}}\geq\cdots$. 
If $\lambda_{j}=0$ for some $j\geq0$, then we complete the proof. 
If $\lambda_{j}>0$ for all $j\geq0$, then we have to check ${\rm lim}_{j\to \infty}\lambda_{j}=0$. 
To see this, we put $\lambda={\rm lim}_{j\to \infty}\lambda_{j}$.
By construction we have $\lambda\neq \lambda_{j}$ for any $j$. 
If $\lambda>0$, then $(X,\Delta+\lambda H)$ has a log minimal model over $Z$ by hypothesis. 
Then we see that $\lambda=\lambda_{j}$ for some $j$ (cf.~\cite[Theorem 4.1(iii)]{birkar-flip}), which is a contradiction. 
Therefore ${\rm lim}_{j\to \infty}\lambda_{j}=0$ and so we are done. 
\end{proof}

Finally we introduce the following useful lemma. 
Roughly speaking it is $\mathbb{R}$-boundary divisor version of \cite[Lemma 2.10]{haconxu-lcc}. 

\begin{lem}\label{lembirequiv}
Let $\pi:X \to Z$ be a projective morphism of normal quasi-projective varieties, and let $(X,\Delta)$ be a log canonical pair. 
Let $(Y,\Gamma)$ be a log canonical pair such that we have a projective birational morphism $f:Y\to X$ and $K_{Y}+\Gamma=f^{*}(K_{X}+\Delta)+E$ for an effective $f$-exceptional divisor $E$. 

Then $(X,\Delta)$ has a weak lc model (resp.~a log minimal model, a good minimal model) over $Z$ if and only if  $(Y,\Gamma)$ has a weak lc model (resp.~a log minimal model, a good minimal model) over $Z$. 
\end{lem}

\begin{proof}
By Lemma \ref{lemweakmin} it is sufficient to show the lemma in the weak lc model case, and moreover we can freely take log smooth models or dlt blow-ups.

First we take a dlt blow-up $\phi_{X}:\bigl(\widetilde{X},\widetilde{\Delta}\bigr) \to (X,\Delta)$ and next we take a common resolution $\phi_{Y}:\widetilde{Y}\to Y$ and $\widetilde{f}:\widetilde{Y}\to \widetilde{X}$ such that $\phi_{Y}$ and $\widetilde{f}$ are log resolutions of $(Y,\Gamma)$ and $\bigl(\widetilde{X},\widetilde{\Delta}\bigr)$ respectively. 
Let $\bigl(\widetilde{Y},\widetilde{\Gamma}\bigr)$ be a log smooth model of $(Y,\Gamma)$. 
%$$\xymatrix{(Y,\Gamma') \ar[d]_{\phi_{Y}}\ar[r]^{f'}&(X'\Delta')\ar[d]\ar[d]^{\phi_{X}}\\(Y,\Gamma)\ar[r]_{f}&(X,\Delta)}$$
Then for some $\phi_{Y}$-exceptional divisor $F$ we have 
\begin{equation*}
\begin{split}
K_{\widetilde{Y}}+\widetilde{\Gamma}=\phi_{Y}^{*}(K_{Y}+\Gamma)+F=&(f\circ \phi_{Y})^{*}(K_{X}+\Delta)+\phi_{Y}^{*}E+F\\
=&{\widetilde{f}}^{*}\bigl(K_{\widetilde{X}}+\widetilde{\Delta}\bigr)+\phi_{Y}^{*}E+F. 
\end{split}
\end{equation*}
We show that $\phi_{Y}^{*}E+F$ is $\widetilde{f}$-exceptional. 
By construction  it is clearly $(\phi_{X}\circ \widetilde{f})$-exceptional. 
On the other hand, every $\phi_{X}$-exceptional prime divisor $D$ satisfies $a\bigl(D,\widetilde{X},\widetilde{\Delta}\bigr)=-1$.
Therefore, if $\phi_{Y}^{*}E+F$ is not $\widetilde{f}$-exceptional, we can find a component $E_{0}$ of $\phi_{Y}^{*}E+F$ such that $a\bigl(E_{0},\widetilde{Y},\widetilde{\Gamma}\bigr)<a\bigl(E_{0},\widetilde{X},\widetilde{\Delta}\bigr)=-1$. 
It contradicts the definition of log smooth models.
In this way we see that $\phi_{Y}^{*}E+F$ is $\widetilde{f}$-exceptional. 
Therefore, by replacing $(X,\Delta)$ and $(Y,\Gamma)$ with $\bigl(\widetilde{X},\widetilde{\Delta}\bigr)$ and $\bigl(\widetilde{Y},\widetilde{\Gamma}\bigr)$, we can assume that $(X,\Delta)$ and $(Y,\Gamma)$ are $\mathbb{Q}$-factorial dlt pairs and $\bigl(Y,{\rm Supp}\,f_{*}^{-1}\Delta\cup{\rm Ex}(f)\bigr)$ is log smooth. 

Suppose that $(X,\Delta)$ has a weak lc model $(X',\Delta')$ over $Z$. 
Then we can easily check that $(X',\Delta')$ is also a weak lc model of $(Y,\Gamma)$ over $Z$. 

Conversely suppose that $(Y,\Gamma)$ has a weak lc model over $Z$. 
Note that $K_{Y}+\Gamma$ and $K_{X}+\Delta$ are then pseudo-effective over $Z$. 
We prove that the $(K_{X}+\Delta)$-MMP over $Z$ with scaling of an ample divisor must terminate. 
%Since termination of the relative log MMP is a local property, by restricting $Z$, $(X,\Delta)$ and $(Y,\Gamma)$, we may assume that $Z$ is affine. 
%In particular we may assume that $Z$ is quasi-projective. 
%Then $Z$ is quasi-projective. 
Pick a general ample divisor $A$ on $X$ such that $(X,\Delta+A)$ is lc and both $K_{X}+\Delta+A$ and $A+\llcorner \Delta \lrcorner$ are ample. 
Let $A'\sim_{\mathbb{R},\,Z}A+\llcorner \Delta \lrcorner$ be a general ample divisor over $Z$ and put $\Theta=\Delta-\llcorner \Delta \lrcorner+A'\sim_{\mathbb{R},\,Z}\Delta+A$. 
Then $(X,\Theta)$ is klt and  we may write $K_{Y}+\Psi=f^{*}(K_{X}+\Theta)+E'$, where $E'\geq0$ and $\Psi\geq0$ have no common components and $\Psi_{X}=\Theta$. 
Then $(Y,\Psi)$ is also klt because $\bigl(Y,{\rm Supp}\,f_{*}^{-1}\Delta\cup{\rm Ex}(f)\bigr)$ is log smooth. 
We set 
$$\Theta_{t}=t\Theta+(1-t)\Delta\sim_{\mathbb{R},\,Z}\Delta+tA \quad {\rm and}\quad \Psi_{t}=t\Psi+(1-t)\Gamma.$$
Then we see that $(X,\Theta_{t})$ and $(Y,\Psi_{t})$ are klt for any $0<t\leq1$, $\Theta_{0}=\Delta$, $\Psi_{0}=\Gamma$ and 
$K_{Y}+\Psi_{t}=f^{*}(K_{X}+\Theta_{t})+tE'+(1-t)E.$ 
Note that $K_{X}+\Theta_{t}$ and $K_{Y}+\Psi_{t}$ are big over $Z$ for any $0<t\leq1$. 

Since $(Y,\Gamma)$ has a weak lc model, by Theorem \ref{thmtermi} and Lemma \ref{lemweakmin}, the $(K_{Y}+\Gamma)$-MMP over $Z$ with scaling of an ample divisor terminates with a log minimal model $(Y',\Gamma_{Y'})$. 
Then it is also the $(K_{Y}+\Psi_{t_{0}})$-MMP for a sufficiently small $t_{0}>0$. 
Therefore the $\bigl(K_{Y'}+(\Psi_{t_{0}})_{Y'}\bigr)$-MMP over $Z$ with scaling of an ample divisor terminates with a log minimal model $\bigr(Y'',(\Psi_{t_{0}})_{Y''}\bigr)$ over $Z$ (cf.~\cite[Corollary 1.4.2]{bchm}). 
Since $t_{0}>0$ is sufficiently small, $K_{Y''}+\Gamma_{Y''}$ is also nef over $Z$. 
Thus $K_{Y''}+(\Psi_{t})_{Y''}$ is nef over $Z$ for any $0\leq t\leq t_{0}$. 

Now run the $(K_{X}+\Delta)$-MMP over $Z$ with scaling of $A$
$$(X=X^{0},\Delta=\Delta_{X^{0}})\dashrightarrow\cdots \dashrightarrow (X^{i},\Delta_{X^{i}})\dashrightarrow \cdots$$
and set $\lambda_{i}={\rm inf}\{\nu\geq0 \mid K_{X^{i}}+\Delta_{X^{i}}+\nu A_{{X^{i}}}{\rm \;is\;nef\;over\;}Z\}.$ 
If $\lambda_{i}=0$ for some $i\geq 0$, there is nothing to prove. 
So suppose that $\lambda_{i}>0$ for any $i\geq0$. 
Then ${\rm lim}_{i\to \infty}\lambda_{i}=0$ (see \cite[Theorem 4.1(ii)]{birkar-flip}) and thus  we can find $n>0$ such that $\lambda_{n}<\lambda_{n-1}\leq t_{0}$. 
Here recall that $\Delta+tA\sim_{\mathbb{R},\,Z}\Theta_{t}$. 
By construction $\bigl(X^{n},(\Theta_{\lambda_{n-1}})_{X^{n}}\bigr)$ and $\bigl(X^{n},(\Theta_{\lambda_{n}})_{X^{n}}\bigr)$ are weak lc models of $(X,\Theta_{\lambda_{n-1}})$ and $(X,\Theta_{\lambda_{n}})$ respectively. 
Therefore $\bigl(X^{n},(\Theta_{\lambda_{n-1}})_{X^{n}}\bigr)$ and $\bigl(X^{n},(\Theta_{\lambda_{n}})_{X^{n}}\bigr)$ are weak lc models of $(Y,\Psi_{\lambda_{n-1}})$ and $(Y,\Psi_{\lambda_{n}})$ respectively. 
But since $\lambda_{n}<\lambda_{n-1}\leq t_{0}$ we can easily check that $\bigl(Y'',(\Psi_{\lambda_{n-1}})_{Y''}\bigr)$ and $\bigl(Y'',(\Psi_{\lambda_{n}})_{Y''}\bigr)$ are also weak lc models of $(Y,\Psi_{\lambda_{n-1}})$ and $(Y,\Psi_{\lambda_{n}})$ respectively. 
By Remark \ref{remtwoweak}, if we take a common resolution $g:W\to X^{n}$ and $h:W\to Y''$ of the birational map $X^{n}\dashrightarrow Y''$, we have
\begin{equation*}
\begin{split}
&g^{*}\bigl(K_{X^{n}}+(\Theta_{\lambda_{n-1}})_{X^{n}}\bigr)=h^{*}\bigl(K_{Y''}+(\Psi_{\lambda_{n-1}})_{Y''}\bigr)\qquad {\rm and} \\
&g^{*}\bigl(K_{X^{n}}+(\Theta_{\lambda_{n}})_{X^{n}}\bigr)=h^{*}\bigl(K_{Y''}+(\Psi_{\lambda_{n}})_{Y''}\bigr).
\end{split}
\end{equation*}
Since $\lambda_{n}<\lambda_{n-1}$ and $\Delta=\Theta_{0}$ (resp.~$\Gamma=\Psi_{0}$) is represented by an $\mathbb{R}$-linear combination of $\Theta_{\lambda_{n-1}}$ and $\Theta_{\lambda_{n}}$ (resp.~$\Psi_{\lambda_{n-1}}$ and $\Psi_{\lambda_{n}}$), we have 
$$g^{*}\bigl(K_{X^{n}}+\Delta_{X^{n}}\bigr)=h^{*}\bigl(K_{Y''}+\Gamma_{Y''}\bigr).$$
Then $K_{X^{n}}+\Delta_{X^{n}}$ is nef over $Z$
%Since $(X,\Delta)\dashrightarrow (X^{n},\Delta_{X^{n}})$ is a finitely many steps of the $(K_{X}+\Delta)$-MMP over $Z$. $(X^{n},\Delta_{X^{n}})$ is 
and hence we have $\lambda_{n}=0$, which contradicts our assumption. 
So we are done.
\end{proof}

\section{Log minimal model program on fiber spaces}\label{sec2}
In this section we introduce two lemmas and a proposition used in the proof of the main results. 
Lemma \ref{lemsemi-ample} is known to the experts, but we write down the proof for the reader's convenience.
Lemma \ref{lemfiber} and Proposition \ref{prop2.1} play important roles in Section \ref{sec3}. 
%First we study the restriction of relative log MMP over an open subset.
%The following lemma is known to the experts, but we write down the proof for the reader's convenience. 

\begin{lem}\label{lemsemi-ample}
Let $\pi:X\to Z$ be a projective morphism from a normal variety to a variety and let $(X,\Delta)$ be a $\mathbb{Q}$-factorial lc pair. 
Assume that there exists an open subset $U\subset Z$ such that $K_{X}+\Delta$ is semi-ample over $U$. 
Pick a sufficiently small positive real number $t$ and let $\Delta'$ be an $\mathbb{R}$-divisor, which may not be effective, such that $(X,\Delta+\Delta')$ is lc.

Then for any finitely many steps of the $(K_{X}+\Delta+t\Delta')$-MMP %over $Z$
$$(X=X^{0},\Delta+t\Delta'=\Delta_{X^{0}}+t\Delta'_{X^{0}})\dashrightarrow \cdots\dashrightarrow (X^{n},\Delta_{X^{n}}+t\Delta'_{X^{n}})$$ 
over $Z$, $K_{X^{n}}+\Delta_{X^{n}}$ is semi-ample over $U$.  
\end{lem}

\begin{proof}
We denote the induced morphism $X^{i}\to Z$ by $\pi_{i}$. 

Since $K_{X}+\Delta$ is semi-ample over $U$, there is a normal variety $V$ projective over $U$ and a contraction $\varphi:\pi^{-1}(U)\to V$ over $U$ such that $(K_{X}+\Delta)\!\!\mid_{\pi^{-1}(U)}\sim_{\mathbb{R},\,U}\varphi^{*}A$ for some general ample divisor $A$ over $U$. 
Since $A$ is general, we can write $A=\sum_{j=1}^{m}\alpha_{j}A_{j}$, where $\alpha_{j}>0$ and $A_{j}$ are ample Cartier divisors over $U$. 
We set $\alpha=\sum_{j=1}^{m}\alpha_{j}$. 
Since $t$ is sufficiently small, we may assume that $(1-t)\alpha>2t\cdot{\rm dim}\,X$. 
We put $\varphi_{0}=\varphi:\pi_{0}^{-1}(U)\to V$. 
For every $i>0$ we construct a contraction $\varphi_{i}:\pi_{i}^{-1}(U)\to V$ over $U$ such that $(K_{X^{i}}+\Delta_{X^{i}})\!\!\mid_{\pi_{i}^{-1}(U)}\sim_{\mathbb{R},\,U}\varphi_{i}^{*}A$. 
Then Lemma \ref{lemsemi-ample} follows from the case $i=n$. 

Let $f:X \to X_{\rm cont}$ be the extremal contraction associated to the first step of the $(K_{X}+\Delta+t\Delta')$-MMP over $Z$, and let $\pi^{-1}(U)\to U_{\rm cont}$ be the restriction of $f$ over $U$. 
Note that $U_{{\rm cont}}\subset X_{\rm cont}$. 
If $\pi^{-1}(U)\simeq U_{\rm cont}$, then $\pi_{1}^{-1}(U)\simeq U_{\rm cont}$ and therefore we can put $\varphi_{1}=\varphi:\pi_{1}^{-1}(U)\to V$. 
Then it is clear that $(K_{X^{1}}+\Delta_{X^{1}})\!\!\mid_{\pi_{1}^{-1}(U)}\sim_{\mathbb{R},\,U}\varphi_{1}^{*}A$. 
So we may assume $\pi^{-1}(U)\not\simeq U_{\rm cont}$. 
Since $(X,\Delta+\Delta')$ is lc over $U$ and $K_{X}+\Delta+\Delta'$ is not nef over $U$, by \cite[Theorem 4.6.2]{fujino-book}, there is a rational curve $C\subset\pi^{-1}(U)$ contracted by $f$ such that 
$0<-\bigl((K_{X}+\Delta+\Delta')\!\!\mid_{\pi^{-1}(U)}\bigr)\!\cdot C\leq 2\,{\rm dim}\,X$. 
%Therefore we have\begin{equation*}\begin{split}0>&(K_{X^{0}}+\Delta+t \Delta')\cdot C\\=&(1-t)(K_{U^{0}}+\Delta\!\!\mid_{U^{0}})\cdot C+t(K_{X^{0}}+\Delta\!\!\mid_{U^{0}}+\Delta'\!\!\mid_{U^{0}})\cdot C\\>&(1-t)\sum_{i=1}^{l}\alpha_{i}(\varphi_{0}^{*}A_{i}\,\cdot\, C)-(1-t)\alpha\end{split}\end{equation*}By construction of $\alpha$ we must have $\sum_{i=1}^{l}\alpha_{i}(\varphi_{0}^{*}A_{i}\,\cdot\, C)=0$, which implies that $K_{X^{0}}+\Delta$ is trivial over $X_{\rm cont}$. 
Then we see that $\bigl((K_{X}+\Delta)\!\!\mid_{\pi^{-1}(U)}\bigr)\cdot \,C=0$ by the standard argument of the length of extremal rays (cf.~\cite[proof of Proposition 3.2 (5)]{birkar-existII}). 
This implies that $K_{X}+\Delta$ is numerically trivial over $X_{\rm cont}$. 
By the cone theorem (cf.~\cite[Theorem 4.5.2]{fujino-book}), we have $K_{X}+\Delta\sim_{\mathbb{R},\,Z}f^{*}B$ for an $\mathbb{R}$-Cartier $\mathbb{R}$-divisor $B$ on $X_{\rm cont}$. 
Then %over $Z$ the divisor $K_{X^{1}}+\Delta_{X^{1}}$ is $\mathbb{R}$-linearly equivalent to the pullback of $B$.
$B$ is semi-ample over $U$. 
Therefore we see that $K_{X^{1}}+\Delta_{X^{1}}$ is semi-ample over $U$ and it induces a contraction $\varphi_{1}:\pi_{1}^{-1}(U)\to V$ over $U$. 

By construction $(K_{X^{1}}+\Delta_{X^{1}})\!\!\mid_{\pi_{1}^{-1}(U)}\sim_{\mathbb{R},\,U}\varphi_{1}^{*}A$ and $(X^{1},\Delta_{X^{1}}+\Delta'_{X^{1}})$ is lc over $U$. 
Now apply the above argument to $\varphi_{1}:\pi_{1}^{-1}(U)\to V$. % and $(K_{X^{1}}+\Delta_{X^{1}})\!\!\mid_{\pi_{1}^{-1}(U)}\sim_{\mathbb{R},\,U}\varphi_{1}^{*}A$.  
By repeating it,  for every $i>0$, we can construct $\varphi_{i}:\pi_{i}^{-1}(U)\to V$ over $U$ such that  $(K_{X^{i}}+\Delta_{X^{i}})\!\!\mid_{\pi_{i}^{-1}(U)}\sim_{\mathbb{R},\,U}\varphi_{i}^{*}A$. 
In this way we see that $K_{X^{n}}+\Delta_{X^{n}}$ is semi-ample over $U$.
%Then Lemma \ref{lemsemi-ample} follows from the case $i=n$. 
So we complete the proof. 
\end{proof}

%The following lemma and proposition play key roles in proof of main results.

\begin{lem}\label{lemfiber}
Let  $\varphi: X\to V$ and $g:V\to Z$ be projective surjective morphisms of normal quasi-projective varieties with connected fibers. 
Let $(X,\Delta)$ be a log canonical pair. 
We assume $\nu(X/V,\,K_{X}+\Delta)=0$. 
Then there is a diagram
 $$\xymatrix{(X,\Delta) \ar[d]_{\varphi}&&(X',\Delta')\ar[d]^{\varphi'}\\V\ar[dr]_{g}&&V'\ar[ll]\ar[dl]^{g'}\\&Z&}$$
where $(X',\Delta')$ is a $\mathbb{Q}$-factorial dlt pair and $X'$ and $V'$ are normal quasi-projective varieties 
such that 
\begin{enumerate}
\item[(i)]
$X'$ is birational to $X$, $V'\to V$ is projective and birational, and $\varphi'$ is a projective surjective morphism with connected fibers, 
%\item[(ii)]$\bigl(\widetilde{X}',\widetilde{\Delta}'\bigr)$ is a log smooth model of $(X,\Delta)$,
%\item[(iii)]the birational map $\bigl(\widetilde{X}',\widetilde{\Delta}'\bigr)\dashrightarrow (X',\Delta')$ is a finitely many steps of the $\bigl(K_{\widetilde{X}'}+\widetilde{\Delta}'\bigr)$-MMP over $V'$, and
%\item[(ii)]for any prime divisor $D$ over $X'$  if $a\bigl(D,X',\Delta')=-1$ then $a\bigl(D,X,\Delta)=-1$, and
\item[(ii)]
 $(X,\Delta)$ has a good minimal model over $Z$ if and only if $(X',\Delta')$ has a good minimal model over $Z$, and
\item[(iii)]
$K_{X'}+\Delta'\sim_{\mathbb{R},\,V'}0$. 
\end{enumerate}
Moreover, if there is an open subset $U\subset Z$ such that all lc centers of $(X,\Delta)$ intersect $(g\circ \varphi)^{-1}(U)$, then all lc centers of $(X',\Delta')$ intersect $(g'\circ \varphi')^{-1}(U)$.  
\end{lem}

\begin{proof}
We prove it with several steps. 

\begin{step4}\label{step1fiber}
By replacing $(X,\Delta)$ with its log smooth model, we can assume that $(X,\Delta)$ is log smooth. 
Since $V$ and $Z$ are both quasi-projective, there exists a projective morphism $\overline{g}:\overline{V}\to \overline{Z}$ 
such that 
\begin{itemize}
\item
$\overline{V}$ (resp.~$\overline{Z}$) is a normal projective variety and $\overline{V}$ (resp.~$\overline{Z}$) contains $V$ (resp.~$Z$) as an open subset, and 
\item
$\overline{g}\!\!\mid_{V}=g$.
\end{itemize}
Now we  apply \cite[Lemma 4.17]{fujino-fund} 
to $\varphi:X\to V$ and $V\subset \overline{V}$. 
Then we can find a projective morphism $\overline{\varphi}:\overline{X}\to \overline{V}$ such that 
\begin{itemize}
\item
$\overline{X}$ is projective and it contains $X$ as an open subset, 
\item
$\overline{\varphi}\!\!\mid_{X}=\varphi$, and
\item
if $\overline{\Delta}$ is the closure of $\Delta$ in $\overline{X}$, then $\bigl(\overline{X}, \overline{\Delta}\bigr)$ is log smooth. 
\end{itemize}
By construction $\overline{g}$ and $\overline{\varphi}$ are contractions, $\overline{g}^{-1}(Z)=V$, $\overline{\varphi}^{-1}(V)=X$ and $\bigl(\overline{X}, \overline{\Delta}\bigr)$ is lc. 
%Let $\overline{\pi}:\overline{X}\to \overline{Z}$ be the composition of $\overline{\varphi}$ and $\overline{g}$. 
%Then $\overline{\pi}$ is a contraction, $\overline{\pi}\!\!\mid_{X}=\pi$ and $\overline{\pi}^{-1}(Z)=X$.
%We note that the pair $\bigl(\overline{X}, \overline{\Delta}\bigr)$ is lc. 
\end{step4}

\begin{step4}\label{step2fiber}
By making minor changes to \cite[proof of Theorem 2.1]{ak}, we can construct a diagram
$$
\xymatrix{
\overline{X} \ar[d]_{\overline{\varphi}}&\overline{Y}\supset U_{\overline{Y}}\ar[l]_{\!\!\!\!\!\!\!\!\!\!\overline{f}}\ar@<-3.1ex>[d]_{\overline{\psi}}\ar@<2.2ex>[d]\\
\overline{V}&\overline{V}'\ar[l]\supset U_{\overline{V}'}\\
}
$$
such that 
\begin{itemize}
\item
$\overline{Y}$ and $\overline{V}'$ are smooth and projective, 
\item
the morphisms $\overline{f}:\overline{Y}\to \overline{X}$ and $\overline{V}'\to \overline{V}$ are birational,
\item
$U_{\overline{Y}}\subset \overline{Y}$ and $U_{\overline{V}'}\subset \overline{V}'$ are open subsets and 
$\overline{\psi}$ is toroidal, and 
\item
There is a simple normal crossing divisor $G$ on $\overline{Y}$ such that ${\overline{f}}^{-1}({\rm Supp}\,\Delta)\cup{\rm Ex}(\overline{f})\subset {\rm Supp}\,G\subset\overline{Y}\backslash U_{\overline{Y}}$.
\end{itemize}
Moreover, by \cite[Remark 4.5]{ak}, we can obtain a birational morphism $\overline{f}':\overline{Y}' \to \overline{Y}$ such that 
\begin{enumerate}
\item[(i)]
the induced morphism $\overline{\psi}\circ \overline{f}':\overline{Y}'\to \overline{V}'$ is a contraction and equidimensional, and
\item[(ii)]
 there is a divisor $G'$ such that ${\overline{f}'}^{-1}({\rm Supp}\,G)\cup {\rm Ex} (\overline{f}')\subset {\rm Supp}\,G'$ and $(\overline{Y}', {\rm Supp}\,G')$ is quasi-smooth (i.e., $(\overline{Y}',{\rm Supp}\,G')$ is toroidal and $\overline{Y}'$ is $\mathbb{Q}$-factorial).
\end{enumerate}
Condition (ii) implies that $\bigl(\overline{Y}',0\bigr)$ is klt and $\bigl(\overline{Y}',{\rm Supp}\,G'\bigr)$ is lc. % (see, for example, \cite[Lemma 5.2]{fujino-toric}). 
Here we regard ${\rm Supp}\,G'$ as a reduced divisor. 
We set $f_{\overline{Y}'}=\overline{f}\circ \overline{f}'$. 
Then we can write 
$$K_{\overline{Y}'}+\overline{\Gamma}'=f_{\overline{Y}'}^{*}(K_{\overline{X}}+\overline{\Delta})+\overline{F},$$
where $\overline{F}\geq0$ and $\overline{\Gamma}' \geq0$ have no common components and $\overline{\Gamma}'_{\overline{X}}=\overline{\Delta}$. 
By construction we see that $\bigl(\overline{Y}',\overline{\Gamma}'\bigr)$ is lc and the image of any lc center of $\bigl(\overline{Y}',\overline{\Gamma}'\bigr)$ on $\overline{X}$ is an lc center of $\bigl(\overline{X},\overline{\Delta}\bigr)$.
\end{step4}

\begin{step4}\label{step3fiber}
Since $\nu(\overline{X}/\overline{V},\,K_{\overline{X}}+\overline{\Delta})=0$, by \cite[Corollary 6.1]{gongyo1}, we have 
$$K_{\overline{Y}'}+\overline{\Gamma}'\sim_{\mathbb{R},\,\overline{V}'}E^{h}+E^{v},$$
where $E^{h}\geq0$ and $E^{v}\geq0$, such that any component of $E^{h}$ dominates $\overline{V}'$ and $E^{v}$ is vertical over $\overline{V}'$. 
Now recall that by condition (ii) fibers of $\overline{\psi}\circ \overline{f}'$ have the same dimension.  
Therefore the image of any component of $E^{v}$ on $\overline{V}'$ is a divisor. 
Since $\overline{V}'$ is smooth, we can consider 
$$\nu_{P}={\rm sup}\{\nu\geq0 \mid E^{v}-\nu (\overline{\psi}\circ \overline{f}')^{*}P {\rm \,\; is \,\; effective}\}$$
for any prime divisor $P$ on $\overline{V}'$. 
Then it is easy to see that $\nu_{P}>0$ with only finitely many prime divisors $P$. 
So we can define $P' =\sum\nu_{P}P$, where $P$ in the summation runs over all prime divisors on $\overline{V}'$. 
Then we see that $E^{v}- (\overline{\psi}\circ \overline{f}')^{*}P'$ is very exceptional over $\overline{V}'$ (for the definition and properties of very exceptional divisors, see
 \cite[Section 3]{birkar-flip}).
By replacing $E^{v}$ with $E^{v}- (\overline{\psi}\circ \overline{f}')^{*}P'$ we can assume that $E^{v}$ is 
very exceptional over $\overline{V}'$. 
%Then $K_{X}+\Delta \sim_{\mathbb{Q},\,W}E^{h}+E'$ and $E'$ is effective. 
%Moreover we see that $E'$ is very exceptional over $W$ (cf.~\cite[Definition 3.1]{birkar-flip}). 

We note that $\nu\bigl(\overline{Y}'/\overline{V}',\,K_{\overline{Y}'}+\overline{\Gamma}'\bigr)=0$ by construction. 
We run the $(K_{\overline{Y}'}+\overline{\Gamma}')$-MMP over $\overline{V}'$ with scaling of an ample divisor. 
Since every component of $E^{h}$ dominates $\overline{V}'$ and $E^{v}$ is very exceptional over $\overline{V}'$, by the argument as in \cite[proof of Theorem 6.2]{has-trivial}, we reach a model $\bigl(\overline{Y}'',\overline{\Gamma}''\bigr)\to \overline{V}'$, where $\overline{\Gamma}''$ is the birational transform of $\overline{\Gamma}'$ on $\overline{Y}''$, such that 
$E^{v}_{\overline{Y}''}\sim_{\mathbb{R},\,\overline{V}'}K_{\overline{Y}''}+\overline{\Gamma}''$ is the limit of movable divisors over $\overline{V}'$. 
Then $E^{v}_{\overline{Y}''}$ is also very exceptional over $\overline{V}'$ because the $(K_{\overline{Y}'}+\overline{\Gamma}')$-MMP only occurs in ${\rm Supp}\,(E^{h}+E^{v})$.  
Thus  $K_{\overline{Y}''}+\overline{\Gamma}''\sim_{\mathbb{R},\,\overline{V}'}0$ by \cite[Lemma 3.3]{birkar-flip} and it is nothing but a good minimal model over $\overline{V}'$. 
%Now can check that $(\overline{Y'},\overline{\Gamma'})$ has a good minimal model over $\overline{V'}$ if and only if $(\overline{X'},\overline{\Delta'})$ has a good minimal model over $\overline{V'}$. Moreover, if all lc centers of $(Y',\Gamma')$ intersect the inverse image of $U$, then all lc centers of $(X',\Delta')$ intersect the inverse image of $U$. 
Let $\overline{g}':\overline{V}'\to \overline{Z}$ be the induced morphism. 
We take a dlt blow-up $\bigl(\overline{X}',\overline{\Delta}'\bigr)\to \bigl(\overline{Y}'',\overline{\Gamma}''\bigr)$ and denote the composition morphism $\overline{X}'\to \overline{Y}''\to \overline{V}'$ by $\overline{\varphi}'$. 
\end{step4}

\begin{step4}\label{step4fiber}
Now we have lc pairs $\bigl(\overline{X},\overline{\Delta}\bigr)$ and $\bigr(\overline{Y}',\overline{\Gamma}'\bigr)$, a $\mathbb{Q}$-factorial dlt pair $\bigl(\overline{X}',\overline{\Delta}'\bigr)$ and varieties $\overline{Z}$, $\overline{V}$ and $\overline{V}'$. 
Moreover each pair or variety has a morphism to $\overline{Z}$. 
Therefore we can take the restrictions of those pairs and varieties over $Z\subset \overline{Z}$, and we denote the restriction of $\bigl(\overline{X},\overline{\Delta}\bigr)$ (resp.~$\bigl(\overline{Y}',\overline{\Gamma}'\bigr)$, $\bigl(\overline{X}',\overline{\Delta}'\bigr)$, $\overline{V}$ and $\overline{V}'$) by 
$(X,\Delta)$ (resp.~$(Y',\Gamma')$, $(X',\Delta')$, $V$ and $V'$). 
Then we obtain the following diagram, 
 $$
\xymatrix{
(X,\Delta) \ar[d]_{\varphi}&(Y',\Gamma')\ar[l]_{f_{Y'}}\ar@{-->}[r]\ar[dr]&(X',\Delta')\ar[d]^{\varphi'}\\
V\ar[dr]_{g}&&V'\ar[ll]\ar[dl]^{g'}\\
&Z&
}
$$
where $f_{Y'}=f_{\overline{Y}'}\!\!\mid_{Y'}$, $\varphi'=\overline{\varphi}'\!\!\mid_{X'}$, $g'=\overline{g}'\!\!\mid_{V'}$ and $(Y',\Gamma') \dashrightarrow (X',\Delta')$ is the restriction of the induced map $\bigr(\overline{Y}',\overline{\Gamma}'\bigr)\dashrightarrow \bigl(\overline{X}',\overline{\Delta}'\bigr)$. 
Then it is clear that $(X',\Delta')$ is $\mathbb{Q}$-factorial dlt, $X'$ and $V'$ are both normal and quasi-projective, and the diagram satisfies conditions (i) and (iii) of Lemma \ref{lemfiber}. 
%Note that by construction $(X',\Delta')$ is $\mathbb{Q}$-factorial dlt and $V'$ is normal quasi-projective.  
Moreover we see that $(X',\Delta')$ has a good minimal model over $Z$ if and only if $(Y',\Gamma')$ has a good minimal model over $Z$. 
By construction we also have 
$K_{Y'}+\Gamma'=f_{Y'}^{*}(K_{X}+\Delta)+F$ for some effective $f_{Y'}$-exceptional divisor $F$ (see Step \ref{step2fiber} of this proof). 
Therefore, by Lemma \ref{lembirequiv}, $(X,\Delta)$ has a good minimal model over $Z$ if and only if $(Y',\Gamma')$ has a good minimal model over $Z$, and thus condition (ii) of Lemma \ref{lemfiber} holds. 
We note that $(X',\Delta')$ is a good minimal model of $(Y',\Gamma')$ over $V'$. 

Finally we show the above diagram satisfies the last assertion of Lemma \ref{lemfiber}. 
Recall that the image of any lc center of $\bigl(\overline{Y}',\overline{\Gamma}'\bigr)$ on $\overline{X}$ is an lc center of $(\overline{X},\overline{\Delta})$, which is stated in the last part of Step \ref{step2fiber} in this proof. 
Pick any prime divisor $D$ over $X'$ such that $a(D,X',\Delta')=-1$. 
We need to show the image of $D$ on $Z$ intersects $U$. 
By construction we have $a(D,Y',\Gamma')=-1$. 
Therefore the image of $D$ on $Y'$ is an lc center of $(Y',\Gamma')$, and thus the image of $D$ on $X$ is an lc center of $(X,\Delta)$. 
By the hypothesis of Lemma \ref{lemfiber}, we see that the image of $D$ on $Z$ intersects $U$, and hence the above diagram satisfies the last assertion of Lemma \ref{lemfiber}. 
Thus we complete the proof. 
\end{step4}
\end{proof}

\begin{rem}\label{remfiber}
Notations as in Lemma \ref{lemfiber}, let $\bigl(\widetilde{X},\widetilde{\Delta}\bigr)$ (resp.~$(\widetilde{X}',\widetilde{\Delta}')$) be a good minimal model of $(X,\Delta)$ (resp.~$(X',\Delta')$) over $Z$, and let $\widetilde{X}\to \widetilde{V}$ (resp.~$\widetilde{X}'\to \widetilde{V}'$) be the contraction over $Z$ induced by $K_{\widetilde{X}}+\widetilde{\Delta}$ (resp.~$K_{\widetilde{X}'}+\widetilde{\Delta}'$). 
Let $p:W\to \widetilde{X}$ and $q:W\to \widetilde{X}'$ be a common resolution of the induced birational map $\widetilde{X}\dashrightarrow \widetilde{X}'$. 
By construction we see that $\bigl(\widetilde{X},\widetilde{\Delta}\bigr)$ and $\bigl(\widetilde{X}',\widetilde{\Delta}'\bigr)$ are two weak lc models of $(Y',\Gamma')$ over $Z$. 
Then we have $p^{*}\bigl(K_{\widetilde{X}}+\widetilde{\Delta}\bigr)=q^{*}\bigl(K_{\widetilde{X}'}+\widetilde{\Delta}'\bigr)$ by Remark \ref{remtwoweak}. 
From this we have $\widetilde{V}\simeq \widetilde{V}'$ and $a\bigl(D,\widetilde{X},\widetilde{\Delta}\bigr)=-1$ if and only if  $a\bigl(D,\widetilde{X}',\widetilde{\Delta}'\bigr)$ for any prime divisor $D$ over $\widetilde{X}$.  
\end{rem}

\begin{prop}\label{prop2.1}
%Let $(X,\Delta)$ be a log canonical pair.
%Let $\varphi: X\to V$ and $g:V\to Z$ be projective surjective morphisms of normal quasi-projective varieties with connected fibers. 
Let $\pi:X \to Z$ be a projective surjective morphism of normal quasi-projective varieties with connected fibers and let $(X,\Delta)$ be a log canonical pair. 
Let $\varphi:X \to V$ be a contraction over $Z$ to a normal quasi-projective variety $V$, which is projective over $Z$, such that 
\begin{itemize}
\item
$\nu(X/V,\,K_{X}+\Delta)\!=0$, 
\item
$\nu(X/Z,\,K_{X}+\Delta)={\rm dim}\,V-{\rm dim}\,Z$, and 
\item
all lc centers of $(X,\Delta)$ dominate $V$. 
\end{itemize}
Then $(X,\Delta)$ has a good minimal model over $Z$. 
Moreover, if $\bigl(\widetilde{X},\widetilde{\Delta}\bigr)$ is a good minimal model of $(X,\Delta)$ over $Z$ and $\widetilde{X}\to \widetilde{V}$ is the contraction over $Z$ induced by $K_{\widetilde{X}}+\widetilde{\Delta}$, then all lc centers of $\bigl(\widetilde{X},\widetilde{\Delta}\bigr)$ dominate $\widetilde{V}$. 
\end{prop}

\begin{proof}
We denote the induced morphism $V\to Z$ by $g$. 
%By taking the Stein factorization of $\varphi$ and $g$, we can assume that $\varphi$ and $g$ are contractions. 
We prove Proposition \ref{prop2.1} with two steps.

\begin{step}\label{step1sub}
In this step we show that we can assume that $X,V$ and $Z$ are projective.

By taking a log resolution of $(X,\Delta)$ we can assume that $(X,\Delta)$ is log smooth. 
Since $V$ and $Z$ are both quasi-projective, there exists a projective morphism $\overline{g}:\overline{V}\to \overline{Z}$ 
such that 
\begin{itemize}
\item
$\overline{V}$ (resp.~$\overline{Z}$) is a normal projective variety and $\overline{V}$ (resp.~$\overline{Z}$) contains $V$ (resp.~$Z$) as an open subset, and 
\item
$\overline{g}\!\!\mid_{V}=g$.
\end{itemize}
Now we  apply \cite[Lemma 4.17]{fujino-fund} 
to $\varphi:X\to V$ and $V\subset \overline{V}$. 
Then we can find a projective morphism $\overline{\varphi}:\overline{X}\to \overline{V}$ such that 
\begin{itemize}
\item
$\overline{X}$ is projective and it contains $X$ as an open subset, 
\item
$\overline{\varphi}\!\!\mid_{X}=\varphi$, and
\item
if $\overline{\Delta}$ is the closure of $\Delta$ in $\overline{X}$, then $\bigl(\overline{X}, \overline{\Delta}\bigr)$ is log smooth. 
\end{itemize}
By construction $\overline{\varphi}$ is a contraction and $\overline{\varphi}^{-1}(V)=X$. 
Let $\overline{\pi}:\overline{X}\to \overline{Z}$ be the composition of $\overline{\varphi}$ and $\overline{g}$. 
Then $\overline{\pi}$ is a contraction, $\overline{\pi}\!\!\mid_{X}=\pi$ and $\overline{\pi}^{-1}(Z)=X$.
We note that the pair $(\overline{X}, \overline{\Delta})$ is lc, 
$\nu\bigl(\overline{X}/\overline{V},\,K_{\overline{X}}+\overline{\Delta}\bigr)=0$ and 
$\nu\bigl(\overline{X}/\overline{Z},\,K_{\overline{X}}+\overline{\Delta}\bigr)={\rm dim}\,\overline{V}-{\rm dim}\,\overline{Z}$. 

If there is an lc center of $\bigl(\overline{X}, \overline{\Delta}\bigr)$ which is contained in 
$\overline{X}\backslash X$, then pick a minimal lc center $C$ of $\bigl(\overline{X}, \overline{\Delta}\bigr)$ contained in $\overline{X}\backslash X$ and 
%Then $C_{i}\cap C_{j}=\emptyset$ for any $i\neq j$ since $(\overline{X}, \overline{B})$ is log smooth. 
%Moreover $C=\amalg_{i=1}^{l}C_{i}\subset \overline{X}\backslash X$ because otherwise $(X,B)$ has an lc center which is vertical over $Z$ and it contradicts to the hypothesis. 
take the blow-up $f_{C}:\overline{X}_{C}\to \overline{X}$ along $C$. 
Then $f_{C}^{-1}(X)=X$ and $\bigl(\overline{X}_{C},\overline{\Delta}_{C}\bigr)$ is log smooth,  where $\overline{\Delta}_{C}$ is the birational transform of $\overline{\Delta}$ on $\overline{X}_{C}$. 
%If there is an lc center of $\bigl(\overline{X}_{C},\overline{\Delta}_{C}\bigr)$ contained in $\overline{X}_{C}\backslash f_{C}^{-1}(X)$, then we replace $\bigl(\overline{X}, \overline{\Delta}\bigr)$ by $\bigl(\overline{X}_{C},\overline{\Delta}_{C}\bigr)$ and repeat the above discussion.  
%then pick a minimal lc center $C'$ of $(\overline{X}_{C},\overline{\Delta}_{C})$ contained in $\overline{X}_{C}\backslash X$ and construct a new log smooth pair as above  by taking the blow-up along $C'$. 
%By local computations of blow-ups, we can check that this recursive procedure eventually stops 
By repeating this blow-up, we obtain a log smooth pair $\bigl(\overline{X}',\overline{\Delta}'\bigr)$ such that $X\subset \overline{X}'$ is an open subset and all lc centers of $\bigl(\overline{X}',\overline{\Delta}'\bigr)$ intersect $X$.
By replacing $\bigl(\overline{X}, \overline{\Delta}\bigr)$ with $\bigl(\overline{X}', \overline{\Delta}'\bigr)$, we can assume that all lc centers of $\bigl(\overline{X}, \overline{\Delta}\bigr)$ intersect $X$. 
%with a pair $(Y,\Gamma)$. Then the constructed birational morphism $f:Y \to X$ satisfies $f^{-1}(X)\simeq X$ and $\Gamma_{X}=\overline{\Delta}$. Therefore we can identify $f^{-1}(X)$ with $X$. In particular we can write $(Y\!\!\mid_{(f\circ\overline{\pi})^{-1}(Z)},\Gamma\!\!\mid_{(f\circ\overline{\pi})^{-1}(Z)})=(X,\Delta)$. By construction we also see that all lc centers of $(Y,\Gamma)$ intersect $X$. 
%Then all lc centers of $\bigl(\overline{X}, \overline{\Delta}\bigr)$ intersects $X$. 
Since all lc centers of $(X,\Delta)$ dominate $V$, which is the hypothesis of the proposition, all lc centers of $\bigl(\overline{X}, \overline{\Delta}\bigr)$ dominate $\overline{V}$. 
%We can also check that $\nu\bigl(\overline{Y}/\overline{V},\,K_{\overline{Y}}+\Gamma\bigr)=0$ and 
%$\nu\bigl(\overline{Y}/\overline{Z},\,K_{\overline{Y}}+\Gamma\bigr)={\rm dim}\,\overline{V}-{\rm dim}\,\overline{Z}$. 
We can also check that it is sufficient to prove Propositoin \ref{prop2.1} with $\overline{\pi}:\bigl(\overline{X}, \overline{\Delta}\bigr)\to \overline{Z}$ and $\overline{\varphi}:\overline{X}\to\overline{V}$ because $\bigl(\overline{X}\!\!\mid_{\overline{\pi}^{-1}(Z)},\overline{\Delta}\!\!\mid_{\overline{\pi}^{-1}(Z)}\bigr)=(X,\Delta)$. 
%Furthermore, if $\bigl(\overline{X}, \overline{\Delta}\bigr)$ has a good minimal model over $\overline{Z}$, then its restriction over $Z$ is a good minimal model of $(X,\Delta)$ over $Z$ because we have $\bigl(\overline{X}\!\!\mid_{\overline{\pi}^{-1}(Z)},\overline{\Delta}\!\!\mid_{\overline{\pi}^{-1}(Z)}\bigr)=(X,\Delta)$. 
%In the rest of the proof we do not use log smoothness of $(X,B)$. By \cite[Lemma 2.13]{has-trivial} and \cite[Corollary 2.14]{has-trivial}, we get a dlt blow-up $f:(Y,\Gamma)\to \bigl(\overline{X},\overline{B}\bigr)$ with $\Gamma=\Gamma'+\Gamma''$, where $\Gamma''$ is reduced and vertical over $\overline{Z}$ and all lc centers of $(Y,\Gamma')$ dominate $\overline{Z}$. Then we see that $f({\rm Supp}\,\Gamma'')\subset \overline{X}\backslash X$ because otherwise $(X,B)$ has an lc center which is vertical over $Z$ and it contradicts to the hypothesis. Thus we have $\Gamma\!\!\mid_{f^{-1}(X)}=\Gamma'\!\!\mid_{f^{-1}(X)}$, and restriction of $(Y,\Gamma')$ to $f^{-1}(X)$ is a dlt blow-up of $(X,B)$. From those facts $f\circ\overline{\pi}:(Y,\Gamma')\to \overline{Z}$ satisfies all the hypotheses of the proposition. Furthermore, if $(Y,\Gamma')$ has a good minimal model over $\overline{Z}$, then restriction of it to the inverse image of $Z$ is a good minimal model of $(X,B)$ because $(\overline{X}\!\!\mid_{\overline{\pi}^{-1}(Z)},\overline{B}\!\!\mid_{\overline{\pi}^{-1}(Z)})=(X,B)$ by construction. 

In this way we can replace $\pi:(X,\Delta)\to Z$ and $\varphi:X\to V$ with $\overline{\pi}:\bigl(\overline{X}, \overline{\Delta}\bigr)\to \overline{Z}$ and $\overline{\varphi}:\overline{X}\to\overline{V}$, 
and we may assume that $X,\,V$ and $Z$ are projective. 
In the rest of proof we do not use the log smoothness of $(X,\Delta)$. 
%In the rest of the proof we do not use the assumption that $(X,B)$ is log smooth. 
\end{step}

\begin{step}\label{step2sub}
We apply Lemmal \ref{lemfiber}. 
Notations as in Lemmal \ref{lemfiber}, we can check that  all lc centers of $(X',\Delta')$ dominate $V'$ by construction (recall that $V'\to V$ is birational and $(X',\Delta')$ is a good minimal model over $V'$ of a higher model of $(X,\Delta)$).
Moreover, by Remark \ref{remfiber}, it is sufficient to show the last assertion of Proposition \ref{prop2.1} for a good minimal model $\bigl(\widetilde{X}',\widetilde{\Delta}'\bigr)$ of $(X',\Delta')$ over $Z$ and the contraction $\widetilde{X}'\to \widetilde{V}'$. 
We can also check that $\nu(X'/Z,K_{X'}+\Delta')={\rm dim}\,V'-{\rm dim}\,Z$. 
%\begin{equation*}\begin{split}\nu(X'/Z,K_{X'}+\Delta')=&\nu(X/Z, K_{X}+\Delta)\\=&{\rm dim}\,V-{\rm dim}\,Z={\rm dim}\,V'-{\rm dim}\,Z.\end{split}\end{equation*}
%by construction (recall that $(X',\Delta')$ is constructed by taking a higher model of $(X,\Delta)$ and running the log MMP over $V'$).  
Therefore we can replace $\pi:(X,\Delta)\to Z$ and $\varphi:X\to V$ with $(X',\Delta')\to Z$ and $\varphi':X'\to V'$, and we may assume that $K_{X}+\Delta\sim_{\mathbb{R},\,V}0$.

Now we carry out the argument of \cite[proof of Proposition 3.3]{birkarhu-arg} or \cite[proof of Proposition 4.3]{has-trivial}.
By \cite[Corollary 3.2]{fg-bundle} there exists a klt pair $(V,\Theta)$ such that $K_{X}+\Delta\sim_{\mathbb{R}}\varphi^{*}(K_{V}+\Theta)$. 
Then $K_{V}+\Theta$ is big over $Z$ by hypothesis. 
Therefore $(V,\Theta)$ has a good minimal model $h:(V,\Theta)\dashrightarrow (V'',\Theta_{V''})$ over $Z$  (cf.~\cite[Corollary 1.4.2]{bchm}). 
%$$\xymatrix{(X,B) \ar[d]_{\pi}&\bigl(\widetilde{X},\widetilde{B}\bigr)\ar@{-->}_{\phi}[l]\ar[d]^{\widetilde{\pi}}\\Z&\widetilde{Z}\ar[l]_{g}}$$such that \begin{enumerate}\item$\phi$ is a birational map and $g$ is a birational morphism, \item$\widetilde{\pi}:\bigl(\widetilde{X},\widetilde{B}\bigr)\to \widetilde{Z}$ is a contraction from a log canonical pair to a normal projective variety such that all lc centers of $\bigl(\widetilde{X},\widetilde{B}\bigr)$ dominates $\widetilde{Z}$.  \item$(X,B)$ has a good minimal model over $Z$ if and only if $\bigl(\widetilde{X},\widetilde{B}\bigr)$ has a good minimal model over $Z$, and \item$K_{\widetilde{X}}+\widetilde{B}\sim_{\mathbb{R},\,\widetilde{Z}}0$. \end{enumerate}By (3) we may prove the existence of good minimal model of $\bigl(\widetilde{X},\widetilde{B}\bigr)$ over $Z$. By (2) and (4) and \cite[Corollary 3.2]{fg-bundle}, there exists a klt pair $\bigl(\widetilde{Z},\Theta\bigr)$ such that $K_{\widetilde{X}}+\widetilde{B}\sim_{\mathbb{R}}\widetilde{\pi}^{*}\bigl(K_{\widetilde{Z}}+\Theta\bigr)$. Moreover $K_{\widetilde{Z}}+\Theta$ is big over $Z$, which follows from (1), and thus $\bigl(\widetilde{Z},\Theta\bigr)$ has a good minimal model $h:\bigl(\widetilde{Z},\Theta\bigr)\dashrightarrow \bigl(\widetilde{Z}',\Theta'\bigr)$ over $Z$ by \cite[Corollary 1.4.2]{bchm}. 
Then $h$ is a birational contraction since $(V,\Theta)$ is klt. 
We take a log smooth model $(Y,\Gamma)$ of $(X,\Delta)$ such that the induced map $\varphi'':Y\dashrightarrow V''$ is a morphism and it factors through a common resolution of $h$. 
%We replace $(X,\Delta)$ by the log smooth model. 
Now we run the $(K_{Y}+\Gamma)$-MMP over $V''$ with scaling of an ample divisor 
$(Y,\Gamma)\dashrightarrow \cdots\dashrightarrow(Y^{i},\Gamma_{Y^{i}})\dashrightarrow \cdots$ and get the following diagram. 
$$
\xymatrix{
(Y,\Gamma) 
\ar[d]%_{\varphi}
\ar[dr]^{\varphi''} 
\ar@{-->}[r]&\cdots \ar@{-->}[r]& (Y^{i},\Gamma_{Y^{i}})\ar[dl]\ar@{-->}[r]&\cdots\\
V\ar@{-->}^{\!\!\!\!\!h}[r]
%\ar[d]_{g}
&V''
%\ar[dl]\\Z
}
$$
%where $(Y,\Gamma)\dashrightarrow \cdots\dashrightarrow(Y^{i},\Gamma_{Y^{i}})\dashrightarrow\cdots$ is the $(K_{Y}+\Gamma)$-MMP over $V'$ with scaling. 
We prove this log MMP terminates. 
Let $V_{0}\subset V$ be the largest open subset such that $h\!\!\mid_{V_{0}}$ is an isomorphism. 
Then $h(V_{0})\subset V''$ is open and ${\rm codim}_{V''}(V''\backslash h(V_{0}))\geq2$. 
By construction the  $(K_{Y}+\Gamma)$-MMP over $V''$ terminates over $h(V_{0})$ with a good minimal model which has relatively trivial log canonical divisor. 
On the other hand, we can write $K_{Y}+\Gamma\sim_{\mathbb{R}}\varphi''^{*}(K_{V''}+\Theta_{V''})+E$ for an effective divisor $E$. 
After finitely many steps we reach a model $\varphi_{n}:(Y^{n},\Gamma_{Y^{n}}\bigr)\to V''$ such that $K_{Y^{n}}+\Gamma_{Y^{n}}\sim_{\mathbb{R}}\varphi_{n}^{*}(K_{V''}+\Theta_{V''})+E_{Y^{n}}$ is the limit of movable divisors over $V''$ and it is trivial over $h(V_{0})$. 
Then we see that $E_{Y^{n}}$ is very exceptional over $V''$ since ${\rm codim}_{V''}(V''\backslash h(V_{0}))\geq2$, and therefore $E_{Y^{n}}=0$.
Then the model $\varphi_{n}:(Y^{n},\Gamma_{Y^{n}}\bigr)\to V''$ satisfies $K_{Y^{n}}+\Gamma_{Y^{n}}\sim_{\mathbb{R}}\varphi_{n}^{*}(K_{V''}+\Theta_{V''})$. 
Because $K_{V''}+\Theta_{V''}$ is semi-ample over $Z$,  we see that  $K_{Y^{n}}+\Gamma_{Y^{n}}$ is semi-ample over $Z$. 
Therefore $(Y,\Gamma)$ has a good minimal model over $Z$ and then $(X,\Delta)$ has a good minimal model over $Z$. 
%So we complete the proof. 

Let $\bigl(\widetilde{X},\widetilde{\Delta}\bigr)$ be a good minimal model of $(X,\Delta)$ over $Z$ and $\widetilde{X}\to \widetilde{V}$ is the contraction over $Z$ induced by $K_{\widetilde{X}}+\widetilde{\Delta}$. 
Notations as above, we can then construct the morphism $Y^{n}\to \widetilde{V}$ induced by $K_{Y^{n}}+\Gamma_{Y^{n}}$, and moreover we can  construct a birational morphism $V''\to \widetilde{V}$ over $Z$ induced by $K_{V''}+\Theta_{V''}$. 
We also see that $a\bigl(D,\widetilde{X},\widetilde{\Delta}\bigr)=-1$ implies $a(D,Y^{n},\Gamma_{Y^{n}})=-1$ for any prime divisor $D$ over  $\widetilde{X}$. 
Then we have $a(D,Y,\Gamma)=-1$ and thus $D$ dominates $V$, i.e., the image of $D$ on $V$ is $V$ itself. 
%Then we can check that $D$ dominates $V''$, i.e., the image of $D$ on $V''$ is $V''$ itself. 
Since $V\dashrightarrow V''\to \widetilde{V}$ is birational, $D$ dominates $\widetilde{V}$. 
Therefore all lc centers of $\bigl(\widetilde{X},\widetilde{\Delta}\bigr)$ dominate $\widetilde{V}$ and thus we complete the proof. 
\end{step}
\end{proof}

\section{Proof of main results and other results}\label{sec3}
%Before the proof of Theorem \ref{thm1.3} we prove the following proposition. 
In this section we prove Theorem \ref{thm1.3}, Theorem \ref{thmhaconxu}, Corollary \ref{corlcc}, Theorem \ref{corflip}, Theorem \ref{corhas} and Corollary \ref{cor1.4}. 

\subsection{Proof of Theorem \ref{thmhaconxu}}\label{subsec4.2}
First we prove Theorem \ref{thmhaconxu}. 
%The idea of the proof is similar to that of \cite[Theorem 1.1]{haconxu-lcc} but we do not use \cite[Theorem 1.1]{haconxu-lcc} in the proof. We also note that in the proof we do not use \cite[Theorem 1.1]{birkar-flip}.  \section{Proof of Theorem \ref{thmhaconxu}}

\begin{proof}[Proof of Theorem \ref{thmhaconxu}]
We prove it by induction on the dimension of $X$. 
But we will not use the induction hypothesis until Step \ref{step5.5haconxu}. 

By taking the Stein factorization of $\pi$, we may assume that $\pi$ is a contraction. 
Moreover, taking a dlt blow-up we may assume that $(X,\Delta)$ is $\mathbb{Q}$-factorial dlt.
We put $U^{X}=\pi^{-1}(U)$ and $\Delta_{U^{X}}=\Delta\!\!\mid_{U^{X}}$.  
%By Theorem \ref{thmtermi} for any $\mathbb{Q}$-factorial dlt pair the existence of log minimal model implies the termination of the log MMP. 
%Furthermore, in relative setting, termination of the log MMP and semi-ampleness of $\mathbb{R}$-Cartier $\mathbb{R}$-divisors are local properties. 
%Therefore, by restricting $(X,\Delta)\to Z$ over an affine open subset of $Z$, we can assume that $Z$ is affine. 
%In particular, any variety which is projective over $Z$ is quasi-projective.  

\begin{step3}\label{step0.5haconxu}
In this step we show that we may assume that $X$ and $Z$ are projective. 

We apply the arguments in Step \ref{step1sub} in the proof of Proposition \ref{prop2.1}. 
More precisely, first we replace $(X,\Delta)$ with its log smooth model. 
Next we take compactifications $X\subset \overline{X}$ and $Z\subset \overline{Z}$ so that there exists a contraction $\overline{\pi}:\overline{X}\to \overline{Z}$ such that $\overline{X}$ (resp.~$\overline{Z}$) contains $X$ (resp.~$Z$) as an open subset and $\overline{\pi}\!\!\mid_{X}=\pi$. 
Note that $\overline{\pi}^{-1}(Z)=X$ by construction.
Let $\overline{\Delta}$ be the closure of $\Delta$ in $\overline{X}$. 
Finally we repeat a blow-up of a minimal lc center contained in $\overline{X}\backslash X$. 
Thus we can construct a morphism, which we denote $\overline{\pi}:\bigl(\overline{X},\overline{\Delta}\bigr)\to \overline{Z}$ by abuse of notation, such that 
\begin{itemize}
\item
$\overline{X}$ (resp.~$\overline{Z}$) is a normal projective variety and $\overline{X}$ (resp.~$\overline{Z}$) contains $X$ (resp.~$Z$) as an open subset, 
\item
$\overline{\pi}\!\!\mid_{X}=\pi$, and
\item
$\bigl(\overline{X},\overline{\Delta}\bigr)$ is a log smooth dlt pair such that all lc centers of $\bigl(\overline{X},\overline{\Delta}\bigr)$ intersect $X$. 
\end{itemize}
Then $U\subset Z\subset \overline{Z}$ is an open subset, and by construction it is easy to check that $\bigl(\overline{X}\!\!\mid_{\overline{\pi}^{-1}(U)},\overline{\Delta}\!\!\mid_{\overline{\pi}^{-1}(U)}\bigr)=\bigl(U^{X},\Delta_{U^{X}}\bigr)$. 
Therefore, with the conditions stated above, we can check that $\overline{\pi}:(\overline{X},\overline{\Delta})\to \overline{Z}$ satisfies the hypothesis of Theorem \ref{thmhaconxu}. 
Furthermore, if $\bigl(\overline{X}, \overline{\Delta}\bigr)$ has a good minimal model over $\overline{Z}$, its restriction over $Z$ is a good minimal model of $(X,\Delta)$ over $Z$ because $\bigl(\overline{X}\!\!\mid_{\overline{\pi}^{-1}(Z)},\overline{\Delta}\!\!\mid_{\overline{\pi}^{-1}(Z)}\bigr)=(X,\Delta)$. 

In this way we can replace $\pi:(X,\Delta)\to Z$ with $\overline{\pi}:\bigl(\overline{X}, \overline{\Delta}\bigr)\to \overline{Z}$ 
and we may assume that $X$ and $Z$ are projective. 
In the rest of proof we do not use the log smoothness of $(X,\Delta)$. 
\end{step3}

\begin{step3}\label{step1haconxu}
Since $(U^{X},\Delta_{U^{X}})$ has a good minimal model over $U$, we have $K_{U^{X}}+\Delta_{U^{X}}\sim_{\mathbb{R},\,U}E_{0}$ for an effective divisor $E_{0}$ on $U^{X}$. 
By taking the closure of $E_{0}$ on $X$ and adding the pullback of an appropriate effective Cartier divisor on $Z$ if necessary, we can find $E_{1}\geq 0$ on $X$ such that $K_{X}+\Delta\sim_{\mathbb{R},\,Z}E_{1}$. 
%Then over $Z$ the divisor $K_{X}+\Delta$ is $\mathbb{R}$-linearly equivalent to an effective divisor. 
Then we have $E_{1}\!\!\mid_{U^{X}}\sim_{\mathbb{R},\,U}\!K_{U^{X}}+\Delta_{U^{X}}$ by construction. 
We take the relative Iitaka fibration $X\dashrightarrow V$ of $E_{1}$ over $Z$ and  let $\bigl(\widetilde{X},\widetilde{\Delta}\bigr)$ be a log smooth model of $(X,\Delta)$ such that the induced map $\varphi:\widetilde{X}\dashrightarrow V$ is a morphism. 
%Then we can easily check that $\nu\bigl(\widetilde{X}/Z, K_{\widetilde{X}}+\widetilde{\Delta}\bigr)={\rm dim}\,V-{\rm dim}\,Z$. 
Since we have $K_{X}+\Delta\sim_{\mathbb{R},\,Z}E_{1}$ and $E_{1}\!\!\mid_{U^{X}}\sim_{\mathbb{R},\,U}\!K_{U^{X}}+\Delta_{U^{X}}$, by construction of $\bigl(\widetilde{X},\widetilde{\Delta}\bigr)$ and the relative Iitaka fibration, we can check that %$\nu(X/Z, K_{X}+\Delta)={\rm dim}\,V-{\rm dim}\,Z$ and $\nu(X/V, K_{X}+\Delta)=0$. 
 $\nu\bigl(\widetilde{X}/Z, K_{\widetilde{X}}+\widetilde{\Delta}\bigr)={\rm dim}\,V-{\rm dim}\,Z$ and $\nu\bigl(\widetilde{X}/V, K_{\widetilde{X}}+\widetilde{\Delta}\bigr)=0$. 
Note that we can also check the second equality with \cite[Theorem 6.1]{leh} (see also \cite{leh2} and \cite{eckl}) and the definition of abundant divisors (cf.~\ref{numeri}). % and the definition of log smooth models, we also see that $\nu\bigl(\widetilde{X}/V, K_{\widetilde{X}}+\widetilde{\Delta}\bigr)=0$.  
Since $\bigl(\widetilde{X},\widetilde{\Delta}\bigr)$ is a log smooth model of $(X,\Delta)$, to prove Theorem \ref{thmhaconxu}, it is enough to show the existence of good minimal models of $\bigl(\widetilde{X},\widetilde{\Delta}\bigr)$ over $Z$.
Thus we can replace $\pi:(X,\Delta)\to Z$ with $\bigl(\widetilde{X},\widetilde{\Delta}\bigr)\to Z$ and assume that there is a contraction $\varphi:X\to V$ over $Z$ to a normal projective variety $V$ such that 
\begin{itemize}
\item
$\nu(X/Z,K_{X}+\Delta)={\rm dim}\,V-{\rm dim}\,Z$, and 
\item
$\nu(X/V,K_{X}+\Delta)=0$. 
\end{itemize}
%\end{step3} 
%\begin{step3}\label{step2haconxu}
%In this step we show that we may assume that $K_{X}+\Delta\sim_{\mathbb{R},\,V}0$ and $\Delta=\Delta'+\Delta''$, where $\Delta'\geq0$ and $\Delta''$ is reduced, such that \begin{itemize}\item$\Delta''$ is vertical over $V$, and \itemall lc centers of $(X,\Delta-\Delta'')$ dominate $V$.  \end{itemize}
Now we can apply Lemma \ref{lemfiber}. 
Notations as in Lemma \ref{lemfiber}, we can check that 
$\nu(X/Z,K_{X}+\Delta)
%=\nu(Y'/Z,K_{Y'}+\Gamma')
=\nu(X'/Z, K_{X'}+\Delta')$
and we may replace $\pi:(X,\Delta)\to Z$ and $\varphi:X \to V$ with $(X',\Delta')\to Z$ and $X' \to V'$ (see Step \ref{step4fiber} in the proof of Lemma \ref{lemfiber}).
%Moreover, by simple calculations of relative numerical dimensions, we see that we can replace $V$ by $V'$. 
%Finally, by condition (ii) and the last assertion of Lemma \ref{lemfiber}, we can check that $(X',\Delta')\to Z$ keeps all the conditions of Theorem \ref{thmhaconxu}, that is, $(X',\Delta')$ has a good minimal model over $U\subset Z$ and all lc centers of $(X',\Delta')$ intersect the inverse image of $U$. 
%In this way we can replace $\pi:(X,\Delta)\to Z$ and $\varphi:X\to V$ by $(X',\Delta')\to Z$ and $\varphi':X'\to V'$. 
In this way we may assume that $K_{X}+\Delta\sim_{\mathbb{R},\,V}0$. 
\end{step3}

\begin{step3}\label{step3haconxu}
In this step we construct a dlt blow-up $(Y,\Gamma)\to (X,\Delta)$ which will be used in the rest of the proof. 

We take a dlt blow-up of $(X,\Delta)$ as in \cite[Corollary 2.14]{has-trivial} and replace $(X,\Delta)$ with the dlt model. 
Then we may assume that $\Delta=\Delta'+\Delta''$, where $\Delta'\geq0$ and $\Delta''$ is reduced, such that $\Delta''$ is vertical over $V$ and all lc centers of $(X,\Delta')$ dominate $V$.  
Pick a divisor $\Xi$ on $V$ such that $K_{X}+\Delta\sim_{\mathbb{R}}\varphi^{*}\Xi$. 
Then $\Xi$ is big over $Z$ from our assumption. 
Since $\Delta''$ is vertical over $V$, there is an effective divisor $\Xi'\sim_{\mathbb{R},\,Z}\varphi^{*}\Xi$ such that ${\rm Supp}\,\Xi'\supset {\rm Supp}\,\Delta''$ and $\Xi'$ is vertical over $V$. 
Then $K_{X}+\Delta\sim_{\mathbb{R},\,Z}\Xi'$.

Take a log smooth model $(Y_{0},\Gamma_{0}\bigr)$ of $(X,\Delta)$ such that $f_{0}:Y_{0}\to X$ is a log resolution of $\bigl(X,{\rm Supp}\,\Delta\cup{\rm Supp}\,\Xi'\bigr)$. 
Let $\Gamma''_{0}$ be the reduced divisor which is the sum of all components of $\llcorner \Gamma_{0} \lrcorner$ contained in ${\rm Supp}\,(f_{0}^{*}\Delta'')$, and set $\Gamma'_{0}=\Gamma_{0}-\Gamma''_{0}$. 
Then $\Gamma_{0}=\Gamma'_{0}+\Gamma''_{0}$ and we can check that $\Gamma''_{0}$ is vertical over $V$ and all lc centers of $(Y_{0},\Gamma'_{0})$ dominate $V$. 
We decompose $f_{0}^{*}\Xi'=G_{0}+D_{0}$ where $G_{0}\geq0$ and $D_{0}\geq0$ have no common components and ${\rm Supp}\,G_{0}={\rm Supp}\,\Gamma''_{0}$. 
By construction $(Y_{0},{\rm Supp}\,\Gamma_{0}\cup{\rm Supp}\,D_{0})$ is log smooth and $\llcorner \Gamma_{0} \lrcorner$ and $D_{0}$ have no common components because any component of $D_{0}$ is not a component of $\Gamma''_{0}$ and vertical over $V$. 
Now we run the $(K_{Y_{0}}+\Gamma_{0})$-MMP over $X$ with scaling and get a dlt blow-up $f:(Y,\Gamma)\to (X,\Delta)$, where $\Gamma$ is the birational transform of $\Gamma_{0}$ on $Y$. 
We decompose $\Gamma=\Gamma'+\Gamma''$, where $\Gamma'$ and $\Gamma''$ are the birational transform of $\Gamma'_{0}$ and $\Gamma''_{0}$ on $Y$ respectively. 
Then $\Gamma''$ is vertical over $V$ and all lc centers of $(Y,\Gamma')$ dominate $V$. 
We also decompose $f^{*}\Xi'=G+D$ where $G$ and $D$ are the birational transform of $G_{0}$ and $D_{0}$ on $Y$ respectively. 
By construction we have $K_{Y}+\Gamma\sim_{\mathbb{R},\,Z}G+D$ and   ${\rm Supp}\,G={\rm Supp}\,\Gamma''$. 
Now recall that  $(Y_{0},{\rm Supp}\,\Gamma_{0}\cup{\rm Supp}\,D_{0})$ is log smooth and $\llcorner \Gamma_{0} \lrcorner$ and $D_{0}$ have no common components. 
Since $Y_{0}\dashrightarrow Y$ is a finitely many steps of the $(K_{Y_{0}}+\Gamma_{0})$-MMP,  
%By the same argument as in Step \ref{step3haconxu} in the proof of Theorem \ref{thmhaconxu}, 
we see that $(Y,\Gamma+\epsilon D)$ is dlt for any sufficiently small $\epsilon>0$ and $a(P,Y,\Gamma+\epsilon D)=-1$ if and only if $a(P,Y,\Gamma)=-1$ for any prime divisor $P$ over $Y$. 

Thus we get a dlt blow-up $f:(Y, \Gamma=\Gamma'+\Gamma'')\to (X,\Delta)$ such that
\begin{enumerate}
\item[(I)]
$\Gamma''$ is reduced and vertical over $V$ and $\Gamma'\geq0$ such that all lc centers of $(Y,\Gamma')$ dominate $V$, and
\item[(II)]
$K_{Y}+\Gamma\sim_{\mathbb{R},\,Z}G+D$ where $G\geq0$ and $D\geq0$ have no common components such that 
\begin{enumerate}
\item[(II-a)]
${\rm Supp}\,G={\rm Supp}\;\Gamma''$, and 
\item[(II-b)]
for any sufficiently small $\epsilon>0$, $(Y,\Gamma+\epsilon D)$ is dlt and $a(P,Y,\Gamma+\epsilon D)=-1$ if and only if $a(P,Y,\Gamma)=-1$ for any prime divisor $P$ over $Y$. 
\end{enumerate}
\end{enumerate}
%In the rest of the proof we devote to show the existence of good minimal model of $(Y',\Gamma_{Y'})$ over $Z$. 
%It is easy to check that $\pi\circ f: (Y,\Gamma)\to Z$ satisfies all the conditions of Theorem \ref{thmhaconxu}. 
Moreover, by construction, we can easily check that $K_{Y}+\Gamma \sim_{\mathbb{R},\,V}0$ and $\nu(Y/Z,K_{Y}+\Gamma)={\rm dim}\,V-{\rm dim}\,Z$. 
\end{step3}

%$$\xymatrix{(X,\Delta=\Delta'+\Delta'') \ar[dd]_{\pi}\ar[dr]^{\varphi} &\\&V\ar[dl]\\Z&}$$
\begin{step3}\label{step4haconxu}
Now we obtain $\pi\circ f:(Y,\Gamma)\to Z$ such that
\begin{itemize}
\item
$(Y,\Gamma)$ is $\mathbb{Q}$-factorial dlt,  
\item
there is a contraction $\varphi\circ f:Y\to V$ over $Z$ to  a normal projective variety $V$ such that
\begin{enumerate}
\item[$\bullet$]
$\nu(Y/Z,K_{Y}+\Gamma)={\rm dim}\,V-{\rm dim}\,Z$, and
\item[$\bullet$]
$K_{Y}+\Gamma \sim_{\mathbb{R},\,V}0$, 
\end{enumerate}
\item
we can write $\Gamma=\Gamma'+\Gamma''$ and  $K_{Y}+\Gamma\sim_{\mathbb{R},\,Z}G+D$, where $\Gamma'$, $\Gamma''$, $G$ and $D$ satisfy conditions (I), (II), (II-a) and (II-b) in Step \ref{step3haconxu} in this proof. 
\end{itemize}
Pick $t>0$ sufficiently small. 
By conditions (II-a) and (I), all lc centers of $(Y,\Gamma-tG)$ dominate $V$. 
Moreover we have
\begin{equation*}
\begin{split}
\nu(Y/Z, K_{Y}+\Gamma-tG)
\leq\nu(Y/Z, K_{Y}+\Gamma)=&\,\nu(Y/Z, (1-t)(K_{Y}+\Gamma))\\
=\nu(Y/Z, K_{Y}+\Gamma-tG-tD)
\leq&\,\nu(Y/Z, K_{Y}+\Gamma-tG).
\end{split}
\end{equation*}
Therefore
%\begin{equation*}\begin{split}
$$\nu(Y/Z, K_{Y}+\Gamma-tG)=\nu(Y/Z, K_{Y}+\Gamma)={\rm dim}\,V-{\rm dim}\,Z.$$
%\end{split}\end{equation*}
We also have $\nu(Y/V,K_{Y}+\Gamma-tG)=0$ because $G$ is vertical over $V$. 
Therefore, by Proposition \ref{prop2.1}, $(Y,\Gamma-tG)$ has a good minimal model over $Z$. 
Let $(Y',\Gamma_{Y'}-tG_{Y'}\bigr)$ be a good minimal model of $(Y,\Gamma-tG)$ and let $Y'\to V'$ be the contraction over $Z$ induced by $K_{Y'}+\Gamma_{Y'}-tG_{Y'}$. 
Then all lc centers of $(Y',\Gamma_{Y'}-tG_{Y'})$ dominate $V'$ (see the last assertion of Proposition \ref{prop2.1}). 
This property will be used in Step \ref{step5haconxu}. 
\end{step3}

\begin{step3}\label{step4.5haconxu}
From now on we prove that $(Y,\Gamma)$ has a good minimal model over $Z$. 
In the rest of the proof we do not use $\varphi\circ f:Y\to V$, so we forget this morphism. 
%The basic strategy is the same as Step \ref{step1.5main}--\ref{step4main} in the proof of Theorem \ref{thm1.3}. 
%We only outline the proof. 

Run the $(K_{Y}+\Gamma)$-MMP over $Z$ with scaling of an ample divisor $M$ 
$$(Y,\Gamma)\dashrightarrow (Y^{1},\Gamma_{Y^{1}})\dashrightarrow \cdots \dashrightarrow (Y^{i},\Gamma_{Y^{i}})\dashrightarrow \cdots.$$ 
Then $K_{Y^{i}}+\Gamma_{Y^{i}}$ is semi-ample over $U$ for any $i\gg0$. 
Fix $i\gg0$ such that $K_{Y^{i}}+\Gamma_{Y^{i}}$ is semi-ample over $U$ and $K_{Y^{i}}+\Gamma_{Y^{i}}+\delta M_{Y^{i}}$ is movable over $Z$ for any sufficiently small $\delta>0$. 
Then for any sufficiently small $t'>0$, $(Y^{i},\Gamma_{Y^{i}}+t'D_{Y^{i}})$ is dlt and $(Y^{i},\Gamma_{{Y}^{i}}-t'G_{Y^{i}})$ has a good minimal model over $Z$. 
By replacing $(Y,\Gamma)$ with $(Y^{i},\Gamma_{Y^{i}})$ for some $i\gg0$, we may assume that $K_{Y}+\Gamma$ is semi-ample over $U$ and there is a big divisor $M$ such that $K_{Y}+\Gamma+\delta M$ is movable for any sufficiently small $\delta>0$. 

In Step \ref{step5.5haconxu}  we use the same notation as above  to mean a sequence of steps of the log MMP, but the log MMP in Step \ref{step5.5haconxu} and the above log MMP are clearly different.  
\end{step3}

\begin{step3}\label{step5haconxu}
Pick an infinite sequence  $\{t_{n}\}_{n\geq1}$ of sufficiently small positive real numbers satisfying that 
%$t_{n}<1$ for all $n$ and 
${\rm lim}_{n\to \infty}t_{n}=0$. 
Then, for every $n$, we can run the $(K_{Y}+\Gamma-t_{n}G)$-MMP over $Z$ with scaling and obtain a good minimal model 
$\psi_{n}:(Y,\Gamma-t_{n}G)\dashrightarrow(Y_{n},\Gamma_{Y_{n}}-t_{n}G_{Y_{n}})$. 
Now take $a_{n}>t_{n}$ sufficiently close to $t_{n}$. 
Then $\{a_{n}\}_{n\geq1}$ is an infinite sequence of sufficiently small positive real numbers such that
% $a_{n}<1$, 
${\rm lim}_{n\to \infty}a_{n}=0$ and $\psi_{n}$ is a finitely many steps of the $(K_{Y}+\Gamma-a_{n}G)$-MMP over $Z$. 
Therefore we see that $(Y_{n},\Gamma_{Y_{n}}-a_{n}G_{Y_{n}})$ has a good minimal model over $Z$, and hence we can run the $(K_{Y_{n}}+\Gamma_{Y_{n}}-a_{n}G_{Y_{n}})$-MMP over $Z$ with scaling and get a good minimal model 
$\phi_{n}:(Y_{n},\Gamma_{Y_{n}}-a_{n}G_{Y_{n}})\dashrightarrow(Y'_{n},\Gamma_{Y'_{n}}-a_{n}G_{Y'_{n}})$ over $Z$.  
Then for any $n$ the birational map 
$$\phi_{n}\circ\psi_{n}:(Y,\Gamma-a_{n}G)\dashrightarrow(Y'_{n},\Gamma_{Y'_{n}}-a_{n}G_{Y'_{n}})$$
is a finitely many step of the $(K_{Y}+\Gamma-a_{n}G)$-MMP over $Z$. 
Since $a_{n}$ is sufficiently close to $t_{n}$, for any $u\in[t_{n},a_{n}]$, we can assume that $K_{Y'_{n}}+\Gamma_{Y'_{n}}-uG_{Y'_{n}}$ is semi-ample over $Z$. 

We check that $(Y'_{n},\Gamma_{Y'_{n}})$ satisfies the hypothesis of Theorem \ref{thmhaconxu} for any $n$. 
Note again that $a_{n}$ is sufficiently small for any $n$, and 
recall that 
%Here we recall that $a_{n}$ is sufficiently small for any $n$ and 
$K_{Y}+\Gamma\sim_{\mathbb{R},\,Z}G+D,$
which is condition (II) in Step \ref{step3haconxu} of this proof. 
We put $\widetilde{a}_{n}=a_{n}/(1-a_{n})$. 
Then we have
$$K_{Y'_{n}}+\Gamma_{Y'_{n}}-a_{n}G_{Y'_{n}}\sim_{\mathbb{R},\,Z} (1-a_{n})(K_{Y'_{n}}+\Gamma_{Y'_{n}}+\widetilde{a}_{n} D_{Y'_{n}}).$$
Therefore, by construction of $Y\dashrightarrow Y'_{n}$, we see that $(Y'_{n},\Gamma_{Y'_{n}}+\widetilde{a}_{n}D_{Y'_{n}})$ is dlt. 
In particular $(Y'_{n},\Gamma_{Y'_{n}})$ is dlt. % for any $n$. 
%Since $Y\dashrightarrow Y'_{n}$ is a  finitely many steps of the $(K_{Y}+\Gamma-a_{n}G)$-MMP, we can easily check that the divisor  $\Gamma_{Y'_{n}}=\Gamma'_{Y'_{n}}+\Gamma''_{Y'_{n}}$ satisfies conditions (I), (II) and (II-a) stated in Step \ref{step1main} in this proof. 
Pick a prime divisor $P$ over $Y'_{n}$ such that  $a(P,Y'_{n},\Gamma_{Y'_{n}})=-1$. 
Then $a(P,Y'_{n},\Gamma_{Y'_{n}}+\widetilde{a}_{n}D_{Y'_{n}})=-1$. 
Since $Y\dashrightarrow Y'_{n}$ is a finitely many steps of the $(K_{Y}+\Gamma-a_{n}G)$-MMP, we have  $a(P,Y,\Gamma+\widetilde{a}_{n}D)=-1$. 
Therefore $a(P,Y,\Gamma)=-1$ by (II-b) in Step \ref{step3haconxu} in this proof. 
Thus we see that the image of $P$ on $Z$ intersects $U$, which implies that all lc centers of $(Y'_{n},\Gamma_{Y'_{n}})$ intersect the inverse image of $U$.  
Moreover, by Lemma \ref{lemsemi-ample}, $K_{Y'_{n}}+\Gamma_{Y'_{n}}$ is semi-ample over $U$. 
Therefore $(Y'_{n},\Gamma_{Y'_{n}})\to Z$ satisfies all the conditions of Theorem \ref{thmhaconxu}. 

Next consider the contraction $Y'_{n} \to V_{n}$ over $Z$ which is induced by $K_{Y'_{n}}+\Gamma_{Y'_{n}}-a_{n} G_{Y'_{n}}$. 
%Note that $V_{n}$ is quasi-projective since $Z$ is affine. 
By the last part of Step \ref{step4haconxu} in this proof, all lc centers of $(Y'_{n}, \Gamma_{Y'_{n}}-a_{n}G_{Y'_{n}})$ dominate $V_{n}$. 
Since $(Y'_{n},\Gamma_{Y'_{n}})$ is dlt, for any $0<\mu\leq a_{n}$, all lc centers of $(Y'_{n}, \Gamma_{Y'_{n}}-\mu G_{Y'_{n}})$ dominate $V_{n}$. 
Furthermore we have 
%\begin{equation*}\begin{split}&\nu(Y'_{n}/Z, K_{Y'_{n}}+\Gamma_{Y'_{n}}-\mu G_{Y'_{n}})\le\nu(Y'_{n}/Z, K_{Y'_{n}}+\Gamma_{Y'_{n}})\\=&\,\nu(Y'_{n}/Z, (1-\mu)(K_{Y'_{n}}\Gamma_{Y'_{n}}))=\nu(Y'_{n}/Z, K_{Y'_{n}}+\Gamma_{Y'_{n}}-\mu G_{Y'_{n}}-\mu D_{Y'_{n}})\\\leq&\,\nu(Y'_{n}/Z, K_{Y'_{n}}+\Gamma_{Y'_{n}}-\mu G_{Y'_{n}}).\end{split}\end{equation*}
\begin{equation*}\begin{split}&\nu(Y'_{n}/Z, K_{Y'_{n}}+\Gamma_{Y'_{n}}-a_{n} G_{Y'_{n}})\le\nu(Y'_{n}/Z, K_{Y'_{n}}+\Gamma_{Y'_{n}}-\mu G_{Y'_{n}})\\\leq&\nu(Y'_{n}/Z, K_{Y'_{n}}+\Gamma_{Y'_{n}})=\nu\bigl(Y'_{n}/Z, (1-a_{n})(K_{Y'_{n}}+\Gamma_{Y'_{n}})\bigr)\\=&\nu\bigl(Y'_{n}/Z, K_{Y'_{n}}+\Gamma_{Y'_{n}}-a_{n} (G_{Y'_{n}}+D_{Y'_{n}})\bigr)\\\leq&\nu(Y'_{n}/Z, K_{Y'_{n}}+\Gamma_{Y'_{n}}-a_{n} G_{Y'_{n}}).\end{split}\end{equation*}
Therefore $\nu(Y'_{n}/Z, K_{Y'_{n}}+\Gamma_{Y'_{n}}-\mu G_{Y'_{n}})=\nu(Y'_{n}/Z, K_{Y'_{n}}+\Gamma_{Y'_{n}}-a_{n} G_{Y'_{n}})$ and thus
$$\nu(Y'_{n}/Z, K_{Y'_{n}}+\Gamma_{Y'_{n}}-\mu G_{Y'_{n}})={\rm dim}\,V_{n}-{\rm dim}\,Z.$$ 
Similarly we see that $\nu(Y'_{n}/V_{n}, K_{Y'_{n}}+\Gamma_{Y'_{n}}-\mu G_{Y'_{n}})=0$. 
Therefore, by Proposition \ref{prop2.1}, $(Y'_{n}, \Gamma_{Y'_{n}}-\mu G_{Y'_{n}})$ has a good minimal model over $Z$. 
%Now we replace $\{a_{n}\}_{n\geq 1}$ with its subsequence so that $\{a_{n}\}_{n\geq 1}$ is strictly decreasing.   

In this way we can obtain an infinite sequence $\{a_{n}\}_{n\geq 1}$ of sufficiently small positive real numbers such that 
\begin{enumerate}
\item[(i)]
%$a_{n}<1$ for all $n$ and 
${\rm lim}_{n\to \infty}a_{n}=0$, and for any $n$, 
\item[(ii)]
there is a finitely many steps of the $(K_{Y}+\Gamma-a_{n}G)$-MMP  over $Z$ to a good minimal model 
$(Y, \Gamma-a_{n}G)\dashrightarrow (Y'_{n}, \Gamma_{Y'_{n}}-a_{n}G_{Y'_{n}})$
such that 
\begin{enumerate}
\item[(ii-a)]
for any $\alpha< a_{n}$ sufficiently close to $a_{n}$, $K_{Y'_{n}}+\Gamma_{Y'_{n}}-\alpha G_{Y'_{n}}$ is semi-ample over $Z$, 
\item[(ii-b)]
$(Y'_{n},\Gamma_{Y'_{n}})$ is dlt, and 
\item[(ii-c)]
$(Y'_{n}, \Gamma_{Y'_{n}}-\mu G_{Y'_{n}})$ has a good minimal model over $Z$ for any $0<\mu \leq a_{n}$.   
\end{enumerate}
\end{enumerate}
Moreover, by condition (ii-a), we have 
\begin{itemize}
\item[(iii)]
for some sufficiently large and divisible $m_{n}>0$ and a general member $H_{n}\sim_{\mathbb{R},\,Z}m_{n}(K_{Y'_{n}}+\Gamma_{Y'_{n}}-a_{n} G_{Y'_{n}})$, $K_{Y'_{n}}+\Gamma_{Y'_{n}}+H_{n}$ is nef over $Z$. 
\end{itemize} 
Since $K_{Y}+\Gamma\sim_{\mathbb{R},\,Z}G+D$,  
we have $K_{Y}+\Gamma-a_{n}G\sim_{\mathbb{R},\,Z}(1-a_{n})G+D$. 
Therefore divisors contracted by the $(K_{Y}+\Gamma-a_{n}G)$-MMP over $Z$ are components of ${\rm Supp}\,(G+D)$, which does not depend on $n$. 
In this way, we can replace $\{a_{n}\}_{n\geq1}$ with its subsequence so that 
\begin{itemize}
\item[(iv)]
all $Y'_{n}$ are isomorphic in codimension one. 
\end{itemize} 
%We remark that the above conditions (i), (ii), (ii-a), (ii-b), (ii-c), (iii) and (iv) correspond to the conditions stated in Step \ref{step2main} in the proof of Theorem \ref{thm1.3}. 
\end{step3}

\begin{step3}\label{step5.25haconxu}
Assume that $(Y'_{1},\Gamma_{Y'_{1}})$, which was constructed in Step \ref{step5haconxu} in this proof, 
has a good minimal model over $Z$. 
Then we can show by the argument as in \cite[Step 4 in the proof of Proposition 5.1]{has-trivial} that $(Y,\Gamma)$ has a good minimal model over $Z$. 
Here we only outline the argument. 

Since $(Y'_{1},\Gamma_{Y'_{1}})$ has a good minimal model over $Z$, we can obtain a finitely many steps of the $(K_{Y'_{1}}+\Gamma_{Y'_{1}})$-MMP over $Z$ to a good minimal model 
$(Y'_{1},\Gamma_{Y'_{1}})\dashrightarrow (Y'',\Gamma_{Y''}).$
Then this log MMP contains only flips because $N_{\sigma}(K_{Y'_{1}}+\Gamma_{Y'_{1}})=0$ over $Z$, which follows from Step \ref{step4.5haconxu} in this proof. 
Fix a sufficiently small positive real number $\epsilon' \ll a_{1}$. 
Then it is easy to check that $(Y'',\Gamma_{Y''}-\epsilon' G_{Y''})$ has a good minimal model over $Z$. 
Thus we can run the $(K_{Y''}+\Gamma_{Y''}-\epsilon' G_{Y''})$-MMP over $Z$ with scaling and get a good minimal model
$(Y'',\Gamma_{Y''}-\epsilon' G_{Y''})\dashrightarrow(Y''',\Gamma_{Y'''}-\epsilon' G_{Y'''}).$ 
We can also check that this log MMP contains only flips because 
\begin{equation*}
\begin{split}
N_{\sigma}&(K_{Y''}+\Gamma_{Y''}-\epsilon' G_{Y''})\\
\leq&\Bigl(1-\frac{\epsilon'}{a_{1}}\Bigr)N_{\sigma}(K_{Y''}+\Gamma_{Y''})+\frac{\epsilon'}{a_{1}}N_{\sigma}(K_{Y''}+\Gamma_{Y''}-a_{1} G_{Y''})=0 
\end{split}
\end{equation*}
over $Z$. 
In this way, since the above two log MMP's contain only flips, we see that $Y'''$ and $Y_{n}$ are isomorphic in codimension one for any $n$. 
Moreover $K_{Y'''}+\Gamma_{Y'''}-
u' G_{Y'''}$ is semi-ample over $Z$ for any $u' \in[0,\epsilon']$ since $\epsilon'$ is sufficiently small. 
By these facts and conditions (i) and (ii) in Step \ref{step5haconxu} in this proof, 
%construction of $\{a_{n}\}_{n\geq1}$ and $Y_{n}$, 
we see that $(Y''',\Gamma_{Y'''}-a_{n}G_{Y'''})$ is a good minimal model of $(Y,\Gamma-a_{n}G)$ over $Z$ for any $n\gg0$. 
%Note that ${\rm lim}_{n\to \infty}a_{n}=0$. 
Let $p:W\to Y$ and $q:W \to Y'''$ be a common resolution of $Y\dashrightarrow Y'''$. 
Then 
$$p^{*}(K_{Y}+\Gamma-a_{n}G)- q^{*}(K_{Y'''}+\Gamma_{Y'''}-a_{n}G_{Y'''})\geq0$$
for any $n\gg0$. 
By considering the limit $n\to \infty$ we have  
$$p^{*}(K_{Y}+\Gamma)- q^{*}(K_{Y'''}+\Gamma_{Y'''})\geq0.$$
In this way we see that $(Y''',\Gamma_{Y'''})$ is a weak lc model of $(Y,\Gamma)$ over $Z$ such that $K_{Y'''}+\Gamma_{Y'''}$ is semi-ample over $Z$. 
Then $(Y,\Gamma)$ has a good minimal model over $Z$ (cf.~Lemma \ref{lemweakmin}). 

Thus we only have to prove the existence of good minimal model of $(Y'_{1},\Gamma_{Y'_{1}})$ over $Z$, and hence we can replace $(Y,\Gamma)$ with $(Y'_{1},\Gamma_{Y'_{1}})$. 
Then, by condition (iii) in Step \ref{step5haconxu} in this proof, there is $m>0$ and a general member $H\sim_{\mathbb{R},\,Z}m(K_{Y}+\Gamma-a G)$ such that  $K_{Y}+\Gamma+H$ is nef over $Z$.

%, we can assume that 
%\begin{enumerate}\item[$\bullet$]$K_{Y}+\Gamma-a G$ is semi-ample for some $0<a\ll1$,  \item[$\bullet$]for some $m>0$ and a general member $H\sim_{\mathbb{R},\,Z}m(K_{Y}+\Gamma-a G),$ $K_{Y}+\Gamma+H$ is nef, and \item[$\bullet$]$(Y,\Gamma- \mu G)$ has a good minimal model over $Z$ for any $0<\mu\leq a$. \end{enumerate}
\end{step3}

\begin{step3}\label{step5.5haconxu}
In this  step we show that $(Y,\Gamma)$ has a log minimal model by applying the standard argument of the spacial termination (see \cite{fujino-sp-ter} and \cite[Step 4 in the proof of Proposition 5.4]{has-trivial}). 

By Lemma \ref{lemmmp} and condition (ii-c) in Step \ref{step5haconxu} in this proof, we can construct a sequence of steps of the $(K_{Y}+\Gamma)$-MMP over $Z$ with scaling of $H$ 
$$(Y,\Gamma)\dashrightarrow (Y^{1},\Gamma_{Y^{1}})\dashrightarrow\cdots \dashrightarrow (Y^{i},\Gamma_{Y^{i}})\dashrightarrow \cdots$$ 
such that if we set
$$\lambda_{i}={\rm inf}\{\nu\geq0\mid K_{Y^{i}}+\Gamma_{Y^{i}}+\nu H_{Y^{i}}{\rm \; is \; nef\;over}\;Z\}$$
for any $i$, then this log MMP terminates or ${\rm lim}_{i\to \infty}\lambda_{i}=0$ even if this log MMP does not terminate. 
%If we set $\lambda={\rm lim}_{i\to \infty}\lambda_{i}$, then we can construct the above MMP with scaling so that $\lambda\neq \lambda_{i}$ for any $i$. 
%Since $(Y,\Gamma-\mu f^{*}B'')$ has a good minimal model over $Z$ for any $0<\mu\leq a$ and 
%$$K_{Y}+\Gamma+\mu' H'=f^{*}(K_{X}+B+\mu' H)\sim_{\mathbb{R},\,Z}(\mu' m+1)f^{*}(K_{X}+B-\widetilde{\mu} B''),$$
%where $\widetilde{\mu}=\frac{\mu' ma}{\mu' m+1}\in(0,1)$, we have $\lambda=0$ (cf.~\cite[Theorem 4.1 (iii)]{birkar-flip}). 
By construction this log MMP only occurs in ${\rm Supp\,G}$ (cf.~\cite[Step 3 in the proof of Proposition 5.4]{has-trivial}). 
In particular this log MMP only occurs in lc centers of $(Y,\Gamma)$. 
Thus it is enough to show this log MMP terminates near all lc centers of $(Y,\Gamma)$. 
%Indeed, $K_{Y}+\Gamma \sim_{\mathbb{R},\,Z}\frac{1}{m}(H'+af^{*}B'')$ and the $(K_{Y_{i}}+\Gamma_{Y_{i}})$-negative extremal ray $R$ intersect to $H'_{Y_{i}}$ positively.  
%Since $(K_{Y_{i}}+\Gamma_{Y_{i}}+\lambda_{i}H_{Y_{i}})\cdot R=0$, 
%we have $\bigl(R \cdot a(f^{*}B'')_{Y_{i}}\bigr)<0$, where $(f^{*}B'')_{Y_{i}}$ is pushforward of $f^{*}B''$ to $Y_{i}$. 
%Thus the $(K_{Y}+\Gamma)$-MMP only occurs in lc centers of $(Y,\Gamma)$. 

Note that by construction $K_{Y^{i}}+\Gamma_{Y^{i}}$ is semi-ample over $U$ for any $i>0$. 
%Let $A_{Y^{i}}$ be the birational transform of $A_{Y}$ on $Y^{i}$. 
Fix a positive integer $d$ and assume that there exists $i_{0}\gg 0$ such that for any $i\geq i_{0}$ the birational map $(Y^{i_{0}},\Gamma_{Y^{i_{0}}})\dashrightarrow (Y^{i},\Gamma_{Y^{i}})$ of the $(K_{Y}+\Gamma)$-MMP is an isomorphism on an open subset containing all lc centers of $(Y^{i_{0}},\Gamma_{Y^{i_{0}}})$ of dimension $d'< d$.   
Pick an lc center $T$ on $Y$ of dimension $d$ such that for any $i$ the birational map $Y\dashrightarrow Y^{i}$ of the $(K_{Y}+\Gamma)$-MMP is birational near the generic point of $T$. 
Let $T^{i}$ be the birational transform of $T$ on $Y^{i}$. 
%Since $K_{Y^{i}}+\Gamma_{Y^{i}}+A_{Y^{i}}\sim_{\mathbb{R},\,Z}0$ for any $i$,  $(Y^{i},\Gamma_{Y^{i}}+A_{Y^{i}})$ is lc for any $i$ and therefore ${\rm Supp}\,A_{Y^{i}}$ does not contain $T^{i}$. 
We define $\Gamma_{T^{i}}$ by the adjunction $K_{T^{i}}+\Gamma_{T^{i}}=(K_{Y^{i}}+\Gamma_{Y^{i}})\!\!\mid_{T^{i}}$.  %and set $A_{T^{i}}=A_{Y^{i}}\!\!\mid_{T^{i}}$. 
%Then $A_{T^{i}}\geq0$ and $K_{T^{i}}+\Gamma_{T^{i}}+A_{T^{i}}\sim_{\mathbb{R},\,Z}0$, and $(T^{i},\Gamma_{T^{i}}+A_{T^{i}})$ is lc because $T^{i}$ is an lc center of a dlt pair $(Y^{i},\Gamma_{Y^{i}})$ and hence it is normal. 
Then for any $i>0$ the pair $(T^{i},\Gamma_{T^{i}})$ is dlt, and for any $0\ll i<j$ the induced birational map $\tau_{ij}:T^{i}\dashrightarrow T^{j}$ is isomorphic in codimension one and 
\begin{equation*}
\begin{split}
K_{T^{i}}+\Gamma_{T^{i}}=&{\rm lim}_{j\to \infty}(K_{Y^{i}}+\Gamma_{Y^{i}}+\lambda_{j}H_{Y^{i}})\!\!\mid_{T^{i}}\\
=&{\rm lim}_{j\to \infty}(\tau_{ij})_{*}^{-1}\bigl((K_{Y^{j}}+\Gamma_{Y^{j}}+\lambda_{j}H_{Y^{j}})\!\!\mid_{T^{j}}\bigr)
\end{split}
\end{equation*}
because ${\rm lim}_{j\to \infty}\lambda_{j}=0$ (see Lemma \ref{lemmmp}). 
Since $K_{Y^{j}}+\Gamma_{Y^{j}}+\lambda_{j}H_{Y^{j}}$ is nef over $Z$, we see that $K_{T^{i}}+\Gamma_{T^{i}}$ is pseudo-effective over $Z$ for any $i\gg 0$. 

By the above discussion, for any lc center $T$ on $Y$ picked as above, there is $i'_{0}\gg0$ such that for all $i\geq i'_{0}$ the morphism $(T^{i},\Gamma_{T^{i}})\to Z$ satisfies the hypothesis of Theorem \ref{thmhaconxu}. 
Therefore, by the induction hypothesis, we may assume that $(T^{i},\Gamma_{T^{i}})$ has a good minimal model over $Z$ for any $i\geq i_{0}$. 
Now apply the standard argument of the spacial termination (cf.~\cite{fujino-sp-ter} and \cite[Step 4 in the proof of Proposition 5.4]{has-trivial}). 
By recalling that the $(K_{Y}+\Gamma)$-MMP only occurs in lc centers of $(Y,\Gamma)$, we can prove the $(K_{Y}+\Gamma)$-MMP terminates with a log minimal model $\bigl(\widetilde{Y},\Gamma_{\widetilde{Y}}\bigr)$ over $Z$.
Then we can check that $\bigl(\widetilde{Y},\Gamma_{\widetilde{Y}}\bigr)\to Z$ satisfies the hypothesis of Theorem \ref{thmhaconxu}. 
%Thus we see that $(X,B)$ also has a log minimal model over $Z$.  
\end{step3}

\begin{step3}\label{step6haconxu}
Finally we prove that $(X,\Delta)$ has a good minimal model, which is equivalent to that the abundance theorem holds for $\bigl(\widetilde{Y},\Gamma_{\widetilde{Y}}\bigr)$ over $Z$. 
Replacing $(X,\Delta) \to Z$ with $\bigl(\widetilde{Y},\Gamma_{\widetilde{Y}}\bigr)\to Z$, we can assume that $K_{X}+\Delta$ is nef over $Z$.  

%Now we use the notations of Step \ref{step1haconxu} in this proof. 
By hypothesis there exists  a contraction $\widetilde{\varphi}:U^{X}\to \widetilde{V}$ over $U$ and a general ample $\mathbb{R}$-divisor $A$ over $U$ such that $(K_{X}+\Delta)\!\!\mid_{U^{X}}\sim_{\mathbb{R},\,U}\widetilde{\varphi}^{*}A$. 
Note that all components of $A$ are ample over $U$ since $A$ is general. 
We write $\Delta=S+B$, where $S$ is the reduced part of $\Delta$ and $\llcorner B\lrcorner=0$. 
Let $\mathcal{H}_{B}$ be the set of boundary $\mathbb{R}$-divisors whose support is contained in ${\rm Supp}\;B$, and let $\mathcal{H}_{A}$ be the set of effective $\mathbb{R}$-divisors whose support is contained in ${\rm Supp}\;A$. 
Note that any $A'\in \mathcal{H}_{A}$ is ample over $U$ since every component of $A$ is ample over $U$.   
Then the set
$$\mathcal{L}_{\rm nef}=\left\{B' \in \mathcal{H}_{B} \mid (X,S+B') {\rm \;is\;lc\; and\;}K_{X}+S+B'{\rm \;is\;nef\;over\;}Z\right\}$$
is a rational polytope by \cite[Theorem 4.7.2]{fujino-book} and it contains $B$.  
Now we consider the set 
$$\mathcal{T}=\left\{ 
(B',A')\in \mathcal{L}_{\rm nef}\times \mathcal{H}_{A}%\in {\rm WDiv}_{\mathbb{R}}(X)
\mid
%\bullet \;K_{X}+S+B'\sim_{\mathbb{R},\,Z}E',\\
(K_{X}+S+B')\!\!\mid_{U^{X}}\sim_{\mathbb{R},\,U}\widetilde{\varphi}^{*}A'
\right\}$$
which contains $(B, A)$. 
In particular it is not empty.
Therefore, by the argument of rational polytopes, we can find real numbers $r_{1},\cdots, r_{l}>0$ and pairs of $\mathbb{Q}$-divisors $(B_{1},A_{1}),\cdots, (B_{l},A_{l})\in \mathcal{T}$ such that 
$$\sum_{i=1}^{l}r_{i}=1, \quad\sum_{i=1}^{l}r_{i}B_{i}=B \quad{\rm and}\quad \sum_{i=1}^{l}r_{i}A_{i}=A.$$ 
Moreover, we can choose $B_{i}$ sufficiently close to $B$ so that $\llcorner B_{i}\lrcorner=0$ and $(X,S+B_{i})$ is dlt for any $i$. 
Then lc centers of $(X,S+B_{i})$ coincide with those of $(X,\Delta)$. 
Therefore $(X,S+B_{i})\to Z$ satisfies all the conditions of Theorem \ref{thmhaconxu}. 
%\cite[Theorem 1.1]{haconxu-lcc} for any $i$. 
Then we can check that $K_{X}+S+B_{i}$ is nef and log abundant over $Z$ by choice of $(B_{i},A_{i})$. 
Since $X$ and $Z$ are projective, by \cite[Theorem 4.12]{fujino-gongyo}, we see that $K_{X}+S+B_{i}$ is semi-ample over $Z$. 
%Note that we can also apply \cite[Theorem 1.1]{haconxu-lcc} directly to $K_{X}+S+B_{i}$ to check semi-ampleness of $K_{X}+S+B_{i}$ over $Z$. 
Because we have $K_{X}+\Delta=\sum_{i=1}^{l}r_{i}(K_{X}+S+B_{i})$, we see that $K_{X}+\Delta$ is semi-ample over $Z$. 
So we are done. 
\end{step3}

\end{proof}

\subsection{Proof of Theorem \ref{thm1.3}}\label{subsec4.25}
%Next we prove Theorem \ref{thm1.3}.
In this subsection we consider a result which is stronger than Theorem \ref{thm1.3}, and by proving it without applying \cite[Theorem 1.1]{haconxu-lcc} or \cite[Theorem 1.1]{birkar-flip}, we give proof of Theorem \ref{thm1.3} avoiding use of \cite[Theorem 1.1]{birkar-flip}. 

We prove the following:
\begin{thm}\label{thmmain2}
Let $\pi:X \to Z$ be a projective morphism of normal quasi-projective varieties and $(X,B)$ be a log canonical pair with a boundary $\mathbb{R}$-divisor $B$ such that $K_{X}+B$ is pseudo-effective over $Z$. 
Assume that there exists an effective $\mathbb{R}$-divisor $A$ on $X$ and an open subset $U\subset Z$ such that 
\begin{itemize}
\item
the pair $\bigl(\pi^{-1}(U),(B+A)\!\!\mid_{\pi^{-1}(U)}\bigr)$ is log canonical, 
\item
$(K_{X}+B+A)\!\!\mid_{\pi^{-1}(U)}\sim_{\mathbb{R},\,U}0$, and 
\item
any lc center of $(X,B)$ intersects $\pi^{-1}(U)$. 
\end{itemize}
Then $(X,B)$ has a good minimal model over $Z$. 
\end{thm}
%We note that in Theorem \ref{thmmain2} the pair $(X,B+A)$ is not necessarily globally lc.  
%Theorem \ref{thmmain2} is clearly stronger than Theorem \ref{thm1.3} because the case $U=Z$ in the theorem is Theorem \ref{thm1.3}. 
%We also note that the theorem immediately follows from Theorem \ref{thm1.3} and Theorem \ref{thmhaconxu}. 
When $U=Z$ in Theorem \ref{thmmain2} this theorem is Theorem \ref{thm1.3}. 
So it is enough to show this theorem.
%From now on we prove the theorem with Theorem \ref{thmhaconxu}. 

%From now on we prove Theorem \ref{thmmain2}. 
%Note that we do not assume $(X,B+A)$ is lc though we can assume it (cf.~Step \ref{step1reduc} in the proof of Lemma \ref{lemreduc}). 
%So we forget this assumption. 

\begin{proof}[Proof of Theorem \ref{thmmain2}]
%By Lemma \ref{lemreduc} we can assume that $X$ and $Z$ are projective. 
We prove it by induction on the dimension of $X$. 
By taking the Stein factorization we can assume that $\pi$ is a contraction.
%The proof is very similar to the proof of Theorem \ref{thmhaconxu}. 
%Note that by the first condition of the theorem $A\!\!\mid_{\pi^{-1}(U)}$ is $\mathbb{R}$-Cartier. 
%We prove Theorem \ref{thmmain2} with several steps. 
\begin{step5}\label{step1app}
In this step we show we may assume that $(X,B)$ is $\mathbb{Q}$-factorial dlt and $X$ and $Z$ are projective. 
%We set $\Delta=B+A$.

By hypothesis there are compactifications $X\subset X_{c}$ and $Z\subset Z_{c}$, where $X_{c}$ (resp.~$Z_{c}$) is normal projective and $X_{c}$ (resp.~$Z_{c}$) contains $X$ (resp.~$Z$) as an open subset, and there is a contraction $\pi_{c}: X_{c}\to Z_{c}$ such that  $\pi_{c}\!\!\mid_{X}=\pi$. 
Then we have $\pi_{c}^{-1}(Z)=X$ and $\pi_{c}^{-1}(U)=\pi^{-1}(U)$ by construction.
Let $B_{c}$ be the closure of $B$ in $X_{c}$. 
We take a log resolution $\phi_{c}:\widetilde{X}_{c}\to X_{c}$ of $\bigl(X_{c},{\rm Supp}\,B_{c}\cup (X_{c}\backslash X)\bigr)$.
%and we mean the restriction of $\phi^{c}$ over $X$ by denoting 
Put $\widetilde{X}=\phi_{c}^{-1}(X)$
and let $\phi:\widetilde{X}\to X$ be the restriction of $\phi_{c}$ over $X$.
Then we can write
%\begin{equation*}\begin{split}
$$K_{\widetilde{X}}+\widetilde{B}=\phi^{*}(K_{X}+B)+F,$$
%\end{split}\end{equation*}
where $\widetilde{B}\geq 0$ and $F\geq 0$ have no common components and $\widetilde{B}_{X}=B$.
%, and similarly $\widetilde{\Delta}\geq 0$ and $F_{\Delta}\geq 0$ have no common components and $\widetilde{\Delta}_{X}=\Delta$. 
Let $\widetilde{B}_{c}$ be the closure of $\widetilde{B}$ in $\widetilde{X}_{c}$. 
Since $(X,B)$ is lc, by construction, $\bigl(\widetilde{X}_{c},\widetilde{B}_{c}\bigr)$ is lc and log smooth. 
%By construction we also see that $\widetilde{B}_{c}\leq\widetilde{\Delta}_{c}$ and $\bigl(\widetilde{X}_{c},\widetilde{\Delta}_{c}\bigr)$ is log smooth. 
By taking more blow-ups if necessary, we can assume that all lc centers of $\bigl(\widetilde{X}_{c},\widetilde{B}_{c}\bigr)$ intersect $\widetilde{X}$ (cf.~Step \ref{step1sub} in the proof of Proposition \ref{prop2.1}). 
Then, by the third condition of Theorem \ref{thmmain2}, we can check that all lc centers of $\bigl(\widetilde{X}_{c},\widetilde{B}_{c}\bigr)$ intersect $(\pi_{c}\circ \phi_{c})^{-1}(U)$. 
%Since any lc center of $\bigl(\widetilde{X}_{c},\widetilde{B}_{c}\bigr)$ is also  an lc center of $\bigl(\widetilde{X}_{c},\widetilde{\Delta}_{c}\bigr)$, 
%all lc centers of $\bigl(\widetilde{X}_{c},\widetilde{B}_{c}\bigr)$ intersect $\widetilde{X}$. 
%Therefore, by construction of $\bigl(\widetilde{X}_{c},\widetilde{B}_{c}\bigr)$ and the third condition of Theorem \ref{thmmain2}, all lc centers of $\bigl(\widetilde{X}_{c},\widetilde{B}_{c}\bigr)$ intersects $(\pi_{c}\circ\phi_{c})^{-1}(U)$. 
We also see that $\bigl(\widetilde{X}_{c},\widetilde{B}_{c}\bigr)\to X_{c}$ and $X\subset X_{c}$ satisfy all the conditions of Theorem \ref{thmhaconxu}. 
Therefore we can run the $\bigl(K_{\widetilde{X}_{c}}+\widetilde{B}_{c}\bigr)$-MMP over $X_{c}$ with scaling and get a good minimal model 
%$\bigl(\widetilde{X}_{c},\widetilde{B}_{c}\bigr)\dashrightarrow (X'_{c},B_{X'_{c}})$
$\phi'_{c}:(X'_{c},B_{X'_{c}})\to X_{c}$, where $B_{X'_{c}}$ is the birational transform of $\widetilde{B}_{c}$ on $X'_{c}$. 
Then all lc centers of $(X'_{c},B_{X'_{c}})$ intersect $(\pi_{c}\circ\phi'_{c})^{-1}(U)$. 
Set $X'={\phi'_{c}}^{-1}(X)$ and let $\phi':X'\to X$ be the restriction of $\phi'_{c}$ over $X$. 
Put  $U^{X'_{c}}=(\pi_{c}\circ \phi_{c}')^{-1}(U)$ and $B_{X'}=(B_{X'_{c}})\!\!\mid_{X'}$. 
%We set $X'={\phi'_{c}}^{-1}(X)$. 
%$$\xymatrix{\widetilde{X}\ar[ddrr]_{\phi}\ar@{}[r]|*{\subset}&\widetilde{X}_{c}\ar[dr]_{\phi_{c}}\ar@{-->}[rr]&&X'_{c}\ar[dl]^{\phi'_{c}}\ar@{}[r]|*{\supset}&X'\ar[ddll]^{\phi'}\\&&X_{c}\ar@{}[d]|*{\cup} &&\\&&X&&}$$
%$$\xymatrix{\widetilde{X}_{c}\ar[dr]_(0.4){\phi_{c}}\ar@{-->}[rr]&&X'_{c}\ar[dl]^(0.4){\phi'_{c}}\ar@{}[r]|*{\supset}&X'\ar[dl]^(0.4){\phi'}\ar@{}[r]|*{\supset}&U^{X'_{c}}\ar[dl]\\&X_{c}\ar@{}[r]|*{\supset} &X\ar@{}[r]|*{\supset} &\pi_{c}^{-1}(U)\\}$$
$$\xymatrix{\widetilde{X}_{c}\ar[dr]_(0.4){\phi_{c}}\ar@{-->}[rr]&&X'_{c}\ar[dl]^(0.4){\phi'_{c}}\ar@{}[r]|*{\supset}&X'\ar[dl]^(0.4){\phi'}\\&X_{c}\ar@{}[r]|*{\supset} &X&\\}$$
% and $\Delta_{X'}=(\Delta_{X'_{c}})\!\!\mid_{X'}$. 
By construction we have
%\begin{equation*}\begin{split}
$K_{X'}+B_{X'}=\phi'^{*}(K_{X}+B).$
%\end{split}\end{equation*}
%where $F_{X'}$ is the birational transform of $F$ on $X'$, which is an effective $\phi'$-exceptional divisor on $X'$.
%We note that $(X'_{c},\Delta_{X'_{c}})$ is $\mathbb{Q}$-factorial dlt and $B_{X'_{c}}\leq \Delta_{X'_{c}}$ by construction.
%We put $U^{X'_{c}}=(\pi_{c}\circ \phi_{c}')^{-1}(U)$. 
On the other hand, by the first condition of Theorem \ref{thmmain2}, 
%Now we recall that
$A\!\!\mid_{\pi_{c}^{-1}(U)}$ is $\mathbb{R}$-Cartier. %, which is stated at the start of this proof. 
Therefore we can define the pullback $(\phi_{c}'\!\!\mid_{U^{X'_{c}}})^{*}(A\!\!\mid_{\pi_{c}^{-1}(U)})$. 
Let $A_{X'_{c}}$ be the closure of $(\phi_{c}'\!\!\mid_{U^{X'_{c}}})^{*}(A\!\!\mid_{\pi_{c}^{-1}(U)})$ in  $X'_{c}$. 
%We set $A_{X'_{c}}=\Delta_{X'_{c}}-B_{X'_{c}}\geq0$ and $A_{X'}=A_{X'_{c}}\!\!\mid_{X'}$. 
%By construction, we see that $(X'_{c},B_{X'_{c}}+A_{X'_{c}})$ and $(X'_{c},B_{X'_{c}})$ are $\mathbb{Q}$-factorial dlt and 
%\begin{equation*}\begin{split}&K_{X'}+B_{X'}+A_{X'}=\phi'^{*}(K_{X}+B+A)\qquad{\rm and}\\ &K_{X'}+B_{X'}=\phi'^{*}(K_{X}+B)+F_{X'}.\end{split}\end{equation*}Now we have a contraction $\pi_{c}\circ\phi'_{c}:X'_{c}\to Z_{c}$. Therefore, if let $U^{X'_{c}}$ be the inverse image of $U\subset Z\subset Z_{c}$ on $X'_{c}$, 
Then over $U$ we have 
$$(K_{X'_{c}}+B_{X'_{c}}+A_{X'_{c}})\!\!\mid_{U^{X'_{c}}}=(\phi_{c}'\!\!\mid_{U^{X'_{c}}})^{*}\bigl((K_{X}
+B+A)\!\!\mid_{\pi_{c}^{-1}(U)}\bigr)\sim_{\mathbb{R},\,U}0.$$
Since we have $K_{X'}+B_{X'}=\phi'^{*}(K_{X}+B)$, by recalling that all lc centers of $(X'_{c},B_{X'_{c}})$ intersect $U^{X'_{c}}$, we see that $(X'_{c},B_{X'_{c}})\to Z_{c}$ and $A_{X'_{c}}$ satisfy all the conditions of Theorem \ref{thmmain2}.  
%because $\pi_{c}^{-1}(U)=\pi^{-1}(U)$ and the second condition of Theorem \ref{thmmain2}. 
%Furthermore all lc centers of $(X'_{c},B_{X'_{c}})$ intersect $U^{X'_{c}}$. 
%Indeed, if $D$ is any prime divisor over $X'_{c}$ such that $a(D,X'_{c},B_{X'_{c}})=-1$, then we have $a\bigl(D,X'_{c},tB_{X'_{c}}+(1-t)\Delta_{X'_{c}}\bigr)=-1$ for any sufficiently small $t>0$ because  $\Delta_{X'_{c}}\geq B_{X'_{c}}$. 
%Moreover we can check that  %$B_{X'_{c}}+A_{X'_{c}}=\Delta_{X'_{c}}$ and
% $\widetilde{X}_{c}\dashrightarrow X'_{c}$ 
 %is a finitely many steps of the $(K_{\widetilde{X}_{c}}+\widetilde{\Delta}_{c})$-MMP, 
% it
 % is a finitely many steps of the $\bigl(K_{\widetilde{X}_{c}}+t\widetilde{B}_{c}+(1-t)\widetilde{\Delta}_{c}\bigr)$-MMP.
%Therefore we have $a\bigl(D,\widetilde{X}_{c},t\widetilde{B}_{c}+(1-t)\widetilde{\Delta}_{c}\bigr)=-1$. 
%Since $\bigl(\widetilde{X}_{c},\widetilde{\Delta}_{c}\bigr)$ is lc, we see that $a\bigl(D,\widetilde{X}_{c}, \widetilde{B}_{c}\bigr)=-1$. 
%Recalling that all lc centers of $\bigl(\widetilde{X}_{c},\widetilde{B}_{c}\bigr)$ intersect $(\pi_{c}\circ\phi_{c})^{-1}(U)$, we see that the image of $D$ on $X'_{c}$ intersects $U^{X'_{c}}$. 
%Thus all lc centers of $(X'_{c},B_{X'_{c}})$ intersect $U^{X'_{c}}$,  
%and therefore $(X'_{c},B_{X'_{c}})\to Z_{c}$ and $A_{X'_{c}}$ satisfy the hypothesis of Theorem \ref{thmmain2}. 
Furthermore, if $(X'_{c},B_{X'_{c}})$ has a good minimal model over $Z_{c}$, then its restriction over $Z\subset Z_{c}$ is a good minimal model of $(X',B_{X'})$ over $Z$. 
Since $K_{X'}+B_{X'}=\phi'^{*}(K_{X}+B)$, by Lemma \ref{lembirequiv}, $(X,B)$ has a good minimal model over $Z$. 

In this way, we can replace $(X,B)\to Z$ and $A$ with $(X'_{c},B_{X'_{c}})\to Z_{c}$ and $A_{X'_{c}}$, and hence we may assume that $(X,B)$ is $\mathbb{Q}$-factorial dlt and $X$ and $Z$ are projective. 
\end{step5}
\begin{step5}\label{step2app}
%We can easily check that the hypothesis of Theorem \ref{thmmain2} still holds after we take dlt blow-ups. 
%So we can freely take dlt blow-ups and replace $(X,B)$ its dlt model.  

Taking a dlt blow-up as in \cite[Corollary 2.14]{has-trivial}, we may assume that we can decompose $B=B'+B''$, where $B'\geq 0$ and $B''$ is a reduced divisor, such that $B''$ is vertical over $Z$ and all lc centers of $(X,B')$ dominate $Z$. 

%We apply the argument as in Step \ref{step3haconxu} in the proof of Theorem \ref{thmhaconxu}. 
Since $(K_{X}+B+A)\!\!\mid_{\pi^{-1}(U)}\sim_{\mathbb{R},\,U}0$ and $K_{X}+B$ is pseudo-effective over $Z$, we see that $A$ is vertical over $Z$. 
Then we have $K_{X}+B\sim_{\mathbb{R},\,Z}E$ for some $E\geq0$ by \cite[Corollary 6.1]{gongyo1}. 
Note that $E$ is vertical over $Z$. % because $K_{X}+B$ is trivial over the general fiber. 
Since $B''$ is vertical over $Z$, by adding an effective Cartier divisor to $E$ if necessary, we can assume that ${\rm Supp}\,E\supset {\rm Supp}\,B''$. 
Then, by the same argument as in Step \ref{step3haconxu} in the proof of Theorem \ref{thmhaconxu}, we obtain a dlt blow-up $f:(Y, \Gamma=\Gamma'+\Gamma'')\to (X,B)$ such that
\begin{enumerate}
\item[(I)]
$\Gamma''$ is reduced and vertical over $Z$ and $\Gamma'\geq0$ such that all lc centers of $(Y,\Gamma')$ dominate $Z$, and
\item[(II)]
$K_{Y}+\Gamma\sim_{\mathbb{R},\,Z}G+D$ where $G\geq0$ and $D\geq0$ have no common components such that 
\begin{enumerate}
\item[(II-a)]
${\rm Supp}\,G={\rm Supp}\;\Gamma''$, and 
\item[(II-b)]
for any sufficiently small $\epsilon>0$, $(Y,\Gamma+\epsilon D)$ is dlt and $a(P,Y,\Gamma+\epsilon D)=-1$ if and only if $a(P,Y,\Gamma)=-1$ for any prime divisor $P$ over $Y$. 
\end{enumerate}
\end{enumerate}
Put $A_{Y}=f^{*}A$. 
Then we can easily check that $(Y,\Gamma)\to Z$ and $A_{Y}$ satisfy the hypothesis of Theorem \ref{thmmain2} and it is sufficient to show the existence of good minimal model of $(Y,\Gamma)$ over $Z$. 

%In the rest of the proof we devote to show the existence of good minimal model of $(Y',\Gamma_{Y'})$ over $Z$. 
\end{step5}

\begin{step5}\label{step3app}
We keep the track of Step \ref{step4haconxu}--\ref{step6haconxu} in the proof of Theorem \ref{thmhaconxu}. 

Pick $t>0$ sufficiently small. 
By conditions (II-a) and (I), we see that all lc centers of $(Y,\Gamma-tG)$ dominate $Z$ for any $t$. 
By construction we also have $\nu(Y/Z,K_{Y}+\Gamma-tG)=0$. 
Therefore, by Proposition \ref{prop2.1},  $(Y,\Gamma-tG)$ has a good minimal model over $Z$. 

From now on we prove $(Y,\Gamma)$ has a good minimal model over $Z$. 
%The basic strategy is the same as Step \ref{step1.25main}--\ref{step4main} in the proof of Theorem \ref{thm1.3}. 
First, as in Step \ref{step4.5haconxu} in the proof of Theorem \ref{thmhaconxu}, we may assume that there is a big divisor $M$ such that $K_{Y}+\Gamma+\delta M$ is movable over $Z$ for any sufficiently small $\delta>0$. 
Note that after this process $(Y,\Gamma)\to Z$ and $A_{Y}$ satisfy the hypothesis of Theorem \ref{thmmain2} and we have $\Gamma=\Gamma'+\Gamma''$ and $K_{Y}+\Gamma\sim_{\mathbb{R},\,Z}G+D$, where $\Gamma'$, $\Gamma''$, $G$ and $D$ satisfy conditions (I), (II), (II-a) and (II-b) stated in Step \ref{step2app} in this proof. 

Next, as in Step \ref{step5haconxu} in the proof of Theorem \ref{thmhaconxu}, we can find an infinite sequence $\{a_{n}\}_{n\geq 1}$ of sufficiently small positive real numbers such that 
\begin{enumerate}
\item[(i)]
${\rm lim}_{n\to \infty}a_{n}=0$, and for any $n$, 
\item[(ii)]
there is a finitely many steps of the $(K_{Y}+\Gamma-a_{n}G)$-MMP  over $Z$ to a good minimal model $(Y, \Gamma-a_{n}G)\dashrightarrow (Y'_{n}, \Gamma_{Y'_{n}}-a_{n}G_{Y'_{n}})$
such that 
\begin{enumerate}
%\item[$\bullet$]$(X'_{n},B_{X'_{n}})$ is lc, 
\item[(ii-a)]
for any $\alpha< a_{n}$ sufficiently close to $a_{n}$, $K_{Y'_{n}}+\Gamma_{Y'_{n}}-\alpha G_{Y'_{n}}$ is semi-ample over $Z$, and
%\end{enumerate}
%\end{enumerate} 
%Moreover, by replacing $\{a_{n}\}_{n \geq 1}$ with its subsequence we may assume that $(Y',\Gamma_{Y'}+a_{n}D_{Y'})$ is dlt for any $n$. 
%By applying the arguments as in Step \ref{step5haconxu} in the proof of Theorem \ref{thmhaconxu},
%Then we can check that for any $n$ 
%\begin{enumerate}
\item[(ii-b)]
$(Y'_{n},\Gamma_{Y'_{n}})$ is dlt. 
\end{enumerate}
\end{enumerate}
Let $A_{Y'_{n}}$ be the birational transform of $A_{Y}$ on $Y'_{n}$. 
As in Step \ref{step5haconxu} in the proof of Theorem \ref{thmhaconxu}, we can also check that $(Y'_{n},\Gamma_{Y'_{n}})\to Z$ and $A_{Y'_{n}}$ satisfy all the conditions of Theorem \ref{thmmain2} for any $n$. 
Moreover, as in the first paragraph of this step, we see that 
\begin{enumerate}
\item[(ii-c)]
$(Y'_{n}, \Gamma_{Y'_{n}}-\mu G_{Y'_{n}})$ has a good minimal model over $Z$ for any $0<\mu \leq a_{n}$.   
\end{enumerate}
By condition (ii-a), we have 
\begin{itemize}
\item[(iii)]
for some sufficiently large and divisible $m_{n}>0$ and a general member $H_{n}\sim_{\mathbb{R},\,Z}m_{n}(K_{Y'_{n}}+\Gamma_{Y'_{n}}-a_{n} G_{Y'_{n}})$, $K_{Y'_{n}}+\Gamma_{Y'_{n}}+H_{n}$ is nef over $Z$. 
\end{itemize}
Finally, replacing $\{a_{n}\}_{n\geq1}$ with its subsequence, we may assume that  
\begin{itemize}
\item[(iv)]
all $Y'_{n}$ are isomorphic in codimension one. 
\end{itemize} 
The above conditions (i), (ii), (ii-a), (ii-b), (ii-c), (iii) and (iv) are corresponding to the conditions in Step \ref{step5haconxu} in the proof of Theorem \ref{thmhaconxu}. 

\end{step5}

\begin{step5}\label{step4app} 
%In this  step we show $(Y,\Gamma_{Y})$ has a log minimal model over $Z$. % by applying the standard argument of the spacial termination (see \cite{fujino-sp-ter} and \cite[Step 4 in the proof of Proposition 5.4]{has-trivial}). 
%But the argument of Step \ref{step2.5main} and Step \ref{step3main} in the proof of Theorem \ref{thm1.3} works with some minor changes because the hypothesis of Theorem \ref{thmmain2} is very similar to the hypothesis of Theorem \ref{thm1.3}. 
%So we only outline the argument. 
As in Step \ref{step5.25haconxu} in the proof of Theorem \ref{thmhaconxu}, we see that it is enough to show the existence of good minimal models of $(Y'_{1},\Gamma_{Y'_{1}})$ over $Z$. 
So we apply Lemma \ref{lemmmp} to $(Y'_{1},\Gamma_{Y'_{1}})$ and we construct a sequence of steps of the $(K_{Y'_{1}}+\Gamma_{Y'_{1}})$-MMP over $Z$ with scaling of $H_{1}$ 
$$(Y'_{1}=Y^{0},\Gamma_{Y'_{1}}=\Gamma_{Y^{0}})\dashrightarrow\cdots \dashrightarrow (Y^{i},\Gamma_{Y^{i}})\dashrightarrow \cdots$$ 
such that if we set
$\lambda_{i}={\rm inf}\{\nu\geq0\mid K_{Y^{i}}+\Gamma_{Y^{i}}+\nu H_{Y^{i}}{\rm \; is \; nef\;over}\;Z\}$
for any $i\geq 0$, then this log MMP terminates or ${\rm lim}_{i\to \infty}\lambda_{i}=0$ even if this log MMP does not terminate.  %for any $i$. 
Let $\pi_{i}:Y^{i}\to Z$ be the induced morphism. 
Then we have $(K_{Y^{i}}+\Gamma_{Y^{i}}+A_{Y^{i}})\!\!\mid_{{\pi_{i}}^{\!\!-1}(U)}\sim_{\mathbb{R},\,U}0$ and  the pair $\bigl({\pi_{i}}^{-1}(U), (\Gamma_{Y^{i}}+A_{Y^{i}})\!\!\mid_{{\pi_{i}}^{\!\!-1}(U)}\bigr)$ is lc. 
Moreover all lc centers of $(Y^{i},\Gamma_{Y^{i}})$ intersect $\pi_{i}^{-1}(U)$. 
Because $(Y^{i},\Gamma_{Y^{i}})$ is dlt, we see that ${\rm Supp}\,A_{Y^{i}}$ does not contain any lc center of $(Y^{i},\Gamma_{Y^{i}})$. 
Then the argument of Step \ref{step5.5haconxu} in the proof of Theorem \ref{thmhaconxu} works with minor changes by the induction hypothesis and the standard argument of the spacial termination. 
Thus we see that the above $(K_{Y'_{1}}+\Gamma_{Y'_{1}})$-MMP over $Z$ terminates with a log minimal model $\bigl(\widetilde{Y},\Gamma_{\widetilde{Y}}\bigr)$. 
Furthermore we can check that $\bigl(\widetilde{Y},\Gamma_{\widetilde{Y}}\bigr)\to Z$ and $A_{\widetilde{Y}}$ satisfy all the conditions of Theorem \ref{thmmain2}. 
\end{step5}

\begin{step5}\label{step5app}
Finally we prove that $(X,B)$ has a good minimal model over $Z$, which is equivalent to that the abundance theorem holds for $\bigl(\widetilde{Y},\Gamma_{\widetilde{Y}}\bigr)$ over $Z$. % by Step \ref{step3main}. 
By replacing $\pi:(X,B)\to Z$ and $A$ with $\bigl(\widetilde{Y},\Gamma_{\widetilde{Y}}\bigr)\to Z$ and $A_{\widetilde{Y}}$, we can assume that $K_{X}+B$ is nef over $Z$. 
Note that $(X,B)$ is $\mathbb{Q}$-factorial dlt. 
Let $\mathcal{H}_{B}$ %\subset {\rm WDiv}_{\mathbb{R}}(X)$
  (resp.~$\mathcal{H}_{A}$%\subset {\rm WDiv}_{\mathbb{R}}(X)$
  ) be the set of effective $\mathbb{R}$-divisors whose support is contained in ${\rm Supp}\,B$ (resp.~${\rm Supp}\,A$). 
Then the set
$$\mathcal{L}_{\rm nef}=\left\{B' \in \mathcal{H}_{B} \mid (X,B') {\rm \;is\;lc\; and\;}K_{X}+B'{\rm \;is\;nef\;over\;}Z\right\}$$
is a rational polytope by \cite[Theorem 4.7.2]{fujino-book} and it contains $B$.
Now we consider the set  
\begin{equation*}\mathcal{T}=
\left\{ 
(B',A')\in \mathcal{L}_{{\rm nef}}\times \mathcal{H}_{A}
 \left|
\begin{array}{l}
\bullet\; \bigl(\pi^{-1}(U),(B'+A')\!\!\mid_{\pi^{-1}(U)}\bigr){\rm \;is\;lc,\; and}\!\!\\
\bullet\;(K_{X}+B'+A')\!\!\mid_{\pi^{-1}(U)}\sim_{\mathbb{R},\,U}0.\\
\end{array}
\right.\right\}
\end{equation*}
which contains $(B,A)$. 
In particular it is not empty.  
Therefore, by the argument of rational polytopes, we can find real numbers $r_{1},\,\cdots, r_{l}>0$ and pairs of $\mathbb{Q}$-divisors $(B_{1},A_{1}),\,\cdots ,(B_{l},A_{l})\in \mathcal{T}$ such that 
$$\sum_{i=1}^{l}r_{i}=1, \quad\sum_{i=1}^{l}r_{i}B_{i}=B \quad{\rm and}\quad \sum_{i=1}^{l}r_{i}A_{i}=A.$$ 
Moreover, we can choose $B_{i}$ sufficiently close to $B$ so that $(X, B_{i})$ are dlt and all lc centers of $(X,B_{i})$ coincide with those of $(X,B)$. 
Then $(X,B_{i})\to Z$ and $A_{i}$ satisfy all the conditions of Theorem \ref{thmmain2}. 
By the second condition of Theorem \ref{thmmain2}% and the induction hypothesis
, we can check that $K_{X}+B_{i}$ is nef and log abundant for any $1\leq i\leq l$. 
Since $X$ and $Z$ are projective, by \cite[Theorem 4.12]{fujino-gongyo}, we see that $K_{X}+B_{i}$ is semi-ample over $Z$. 
Because we have $K_{X}+B=\sum_{i=1}^{l}r_{i}(K_{X}+B_{i})$, we see that $K_{X}+B$ is semi-ample over $Z$. 
So we are done. 
\end{step5}

\end{proof}

\subsection{Proof of other results}\label{subsec4.3}
Finally we show other results. 

\begin{proof}[Proof of Corollary \ref{corlcc}]
The proof of \cite[Corollary 1.2]{haconxu-lcc} works with no changes by using Theorem \ref{thmhaconxu} instead of \cite[Theorem 1.1]{haconxu-lcc}. 
By construction of $(X,\Delta)$ we can easily check that $(X,\Delta)$ satisfies the last assertion of Corollary \ref{corlcc}. 
So we are done. 
\end{proof}

\begin{proof}[Proof of Theorem \ref{corflip}]
We can assume that $W$ is affine because we only have to prove the existence of log canonical model of $(X,\Delta)$ over $W$. 
In particular we may assume that $X$ and $W$ are quasi-projective. 
Since $-(K_{X}+\Delta)$ is ample over $W$, there is a general ample divisor $A\geq0$ such that $(X,\Delta+A)$ is lc and $K_{X}+\Delta+A\sim_{\mathbb{R},\,W}0$. 
Then the result follows from Theorem \ref{thm1.3} and so we are done.
\end{proof}

\begin{proof}[Proof of Theorem \ref{corhas}]
The arguments in \cite{has-trivial} work with minor changes since we can apply Theorem \ref{thm1.3} instead of \cite[Theorem 1.1]{birkar-flip}. 
\end{proof}

\begin{proof}[Proof of Corollary \ref{cor1.4}]
We have $\kappa_{\iota}(X,K_{X}+\Delta)=\nu(X,K_{X}+\Delta)\geq 0$ by hypothesis. 
Thus there exists $E\geq0$ such that $K_{X}+\Delta\sim_{\mathbb{R}}E$ (cf.~\ref{numeri}). 
Take the Iitaka fibration $X\dashrightarrow V$ of $E$ and let $(Y,\Gamma)$ be a log smooth model of $(X,\Delta)$ such that the induced map $Y\dashrightarrow V$ is a morphism.
By \cite[Theorem 6.1]{leh} (see also \cite{leh2} and \cite{eckl}), we see that ${\rm dim}\,V=\nu(Y,K_{Y}+\Gamma)$ and $\nu(Y/V,K_{Y}+\Gamma)=0$ (see also Step \ref{step1haconxu} in the proof of Theorem \ref{thmhaconxu}). 
Then $(Y,\Gamma)$ has a good minimal model by Lemma \ref{lemfiber} (put $Z={\rm Spec}\,\mathbb{C}$ in Lemma \ref{lemfiber}) and Theorem \ref{corhas}, and therefore $(X,\Delta)$ has a good minimal model. 
So we are done. 
\end{proof}

%%%%%%%%%%%%%%%

\end{document}